\documentclass[twoside,a4paper,reqno,11pt]{amsart} 
\usepackage{amsfonts, amsbsy, amsmath, amssymb, latexsym}
\usepackage{mathrsfs}
\usepackage[top=30mm,right=30mm,bottom=30mm,left=30mm]{geometry}
\usepackage{stmaryrd}
\usepackage{bm}

\usepackage{enumerate}
\usepackage{color}
\definecolor{darkblue}{rgb}{0,0,0.8}

\usepackage{hyperref}

\renewcommand{\a}{\alpha}
\renewcommand{\b}{\beta}

\newcommand{\F}{\mathbb{F}_{q}} 
\newcommand{\e}{\epsilon}
 \renewcommand{\L}{\Lambda}
\renewcommand{\l}{\lambda} 

 \renewcommand{\to}{\rightarrow}

 \newcommand{\C}{\mathcal{C}}

\newcommand{\la}{\langle}
\newcommand{\ra}{\rangle}

\newcommand{\leqs}{\leqslant}
\newcommand{\geqs}{\geqslant}

 \newcommand{\vs}{\vspace{3mm}}

\newcommand{\BB}{\mathcal{B}}

\makeatletter
\newcommand{\imod}[1]{\allowbreak\mkern4mu({\operator@font mod}\,\,#1)}
\makeatother

\newtheorem{theorem}{Theorem} 
\newtheorem*{conj*}{Conjecture}

\newtheorem{coroll}{Corollary}

\newtheorem{thm}{Theorem}[section] 
\newtheorem{prop}[thm]{Proposition} 
\newtheorem{lem}[thm]{Lemma}
\newtheorem{cor}[thm]{Corollary}

\theoremstyle{definition}
\newtheorem{rem}[thm]{Remark}
\newtheorem{remk}{Remark}

\newtheorem{defn}[thm]{Definition}

\begin{document}

 \author{Timothy C. Burness}
 \address{T.C. Burness, School of Mathematics, University of Bristol, Bristol BS8 1TW, UK}
 \email{t.burness@bristol.ac.uk}
 
 \author{Michael Giudici}
 \address{M. Giudici, The Centre for the Mathematics of Symmetry and Computation, School of Mathematics and Statistics, The University of Western Australia, Crawley WA 6009, Australia}
 \email{michael.giudici@uwa.edu.au}
 
\title[Locally elusive classical groups]{Locally elusive classical groups}

\begin{abstract}
Let $G$ be a transitive permutation group of degree $n$ with point stabiliser $H$ and let $r$ be a prime divisor of $n$. We say that $G$ is $r$-elusive if it does not contain a derangement of order $r$. The problem of determining the $r$-elusive primitive groups can be reduced to the almost simple case, and the purpose of this paper is to complete the study of $r$-elusivity for almost simple classical groups. Building on our earlier work for geometric actions of classical groups, in this paper we handle the remaining non-geometric actions where $H$ is almost simple and irreducible. This requires a completely different approach, using tools from the representation theory of quasisimple groups. 
\end{abstract}


\date{\today}
\maketitle
 
\section{Introduction}\label{s:intro}

Let $G \leqs {\rm Sym}(\Omega)$ be a transitive permutation group on a finite set $\Omega$ of size at least $2$. By the Orbit-Counting Lemma, $G$ contains elements that act fixed-point-freely on $\Omega$. Such elements are called \emph{derangements}, and their existence turns out to have some interesting applications in many areas of mathematics, such as number theory and topology (see Serre's article \cite{Serre}). 

By a theorem of Fein, Kantor and Schacher \cite{FKS}, $G$ contains a derangement of prime power order (the proof requires the Classification of Finite Simple Groups). In fact, in  most cases, $G$ contains a derangement of prime order, but there are some exceptions, such as the $3$-transitive action of the smallest Mathieu group ${\rm M}_{11}$ on $12$ points. The transitive permutation groups with this property are called \emph{elusive} groups, and they have been extensively studied in recent years (see \cite{CGJKKMN, G, GK, GMPV, xu}, for example). 

A local notion of elusivity was introduced in \cite{BGW}. For a prime divisor $r$ of $|\Omega|$, we say that $G$ is \emph{$r$-elusive} if it does not contain a derangement of order $r$ (so $G$ is elusive if and only if it is $r$-elusive for all such primes $r$). In \cite{BGW}, 
the O'Nan-Scott theorem is used to essentially reduce the problem of determining the $r$-elusive primitive groups to the almost simple case, and the examples with an alternating or sporadic socle are identified in \cite{BGW}. Therefore, it remains to determine the $r$-elusive primitive almost simple groups of Lie type and our goal in this paper is to complete the picture for classical groups (the locally elusive exceptional groups of Lie type will be the subject of a future paper).

Let $G \leqs {\rm Sym}(\Omega)$ be a primitive almost simple classical group over $\mathbb{F}_{q}$ with socle $T$ and point stabiliser $H$. Let $V$ be the natural module for $T$ and write $n=\dim V$ and $q=p^f$, where $p$ is a prime. Note that $H$ is a maximal subgroup of $G$ with $G = HT$. Roughly speaking, Aschbacher's subgroup structure theorem \cite{asch} states that either $H$ belongs to one of eight natural, or \emph{geometric}, subgroup collections (denoted by $\C_1, \ldots, \C_8$), or $H$ is almost simple and acts irreducibly on $V$. The geometric subgroups include the stabilisers of appropriate subspaces and direct sum and tensor product decompositions of $V$ (see \cite[Table 1.4.2]{BG} for a brief description of the subgroups in each $\C_i$ collection). We write $\mathcal{S}$ for the collection of almost simple irreducible subgroups arising in Aschbacher's theorem (see Definition \ref{sdef} for the precise definition of $\mathcal{S}$), and we say that the action of $G$ on $\Omega$ is an \emph{$\mathcal{S}$-action} if $H \in \mathcal{S}$. We will write $S$ for the socle of a subgroup $H \in \mathcal{S}$.

A detailed analysis of the structure, maximality and conjugacy of the geometric subgroups of $G$ is provided in \cite{KL}. This is used extensively in our study of the $r$-elusive geometric actions of almost simple classical groups in \cite{BG} (see \cite[Section 1.5]{BG} for a summary of the main results), which is organised according to Aschbacher's theorem. This  approach relies on the fact that there is a concrete description of the embedding of each geometric subgroup $H$ in $G$, which permits a detailed study of the fusion of the conjugacy classes of $H$ in $G$. This sort of information is not readily available when $H \in \mathcal{S}$ is a non-geometric subgroup of $G$, so a different approach is required. For example, it is not even possible to list all the subgroups in $\mathcal{S}$ of a given classical group, in general (of course, we do not even know the dimensions of all irreducible representations of simple groups).  
However, detailed information is available for the low-dimensional groups with $n \leqs 12$ (see \cite{BHR}), which we use in \cite[Section 6.3]{BG}
to determine the $r$-elusive $\mathcal{S}$-actions for $n \leqs 5$.  
In this paper, our aim is to complete the study of $\mathcal{S}$-actions initiated in \cite{BG} by extending the analysis to all classical groups.

In order to state our main result (Theorem \ref{t:main} below), we need to introduce two subcollections of $\mathcal{S}$, which we denote by the symbols $\mathcal{A}$ and $\mathcal{B}$. A subgroup $H \in \mathcal{S}$ with socle $S$ belongs to the collection $\mathcal{A}$ if and only if $S$ is an alternating group, $q=p$ is prime and $V$ is the fully deleted permutation module for $S$ over $\mathbb{F}_{p}$ (see Table \ref{atab}). The collection $\mathcal{B}$ is recorded in Table \ref{btab}. We need to highlight these specific cases in order to state an important theorem of Guralnick and Saxl  \cite[Theorem 7.1]{GS} on irreducible subgroups of classical groups (see Theorem \ref{gursax}), which plays a key role in our proof of Theorem \ref{t:main}.  

\begin{remk}
Let us make a couple of comments on the cases in Tables \ref{atab} and \ref{btab}.
\begin{itemize}\addtolength{\itemsep}{0.2\baselineskip}
\item[{\rm (i)}] Consider Case $(\mathcal{A}1)$ in Table \ref{atab}. Here $S = A_d$ and $T = {\rm P\Omega}_{n}^{\e}(p)$, where $d \geqs 8$, $p$ is an odd prime and $n = d-\delta$, with $\delta = 2$ if $p$ divides $d$, otherwise $\delta=1$. If $n$ is even then $\e=+$ if and only if
$$\left(\frac{n+1}{p}\right) = (-1)^{\frac{1}{4}n(p-1)}$$
(see Section \ref{s:a}).
\item[{\rm (ii)}] In Table \ref{btab} we write $L(\l)$ for the unique irreducible $\F\hat{S}$-module of highest weight $\l$ (up to quasiequivalence, where $\hat{S}$ denotes the full covering group of $S$), and we follow Bourbaki \cite{Bou} in labelling the fundamental dominant weights $\l_{i}$. We also note that the conditions recorded in the final column of Table \ref{btab} are necessary, but not always sufficient, for the existence and maximality of $H$ in $G$; for the precise conditions, we refer the reader to the relevant tables in \cite[Section 8.2]{BHR}.
\end{itemize}
\end{remk}

Let $r \neq p$ be a prime and let $i \geqs 1$ be minimal such that $r$ divides $q^i-1$. As above, let $n$ be the dimension of the natural module for $T$ and set  
\begin{equation}\label{e:eqcc}
c = \left\{\begin{array}{ll}
2i & \mbox{if $i$ is odd and $T \neq {\rm PSL}_n(q)$} \\
i/2 &  \mbox{if $i\equiv 2\imod{4}$ and $T = {\rm PSU}_n(q)$} \\
i & \mbox{otherwise}
\end{array}\right.
\end{equation}
as in \cite{BG}. We also introduce the following notation:
\begin{equation}\label{e:kop}
\mbox{$\kappa(T,r)$ is the number of conjugacy classes of subgroups of order $r$ in $T$}
\end{equation}
and we highlight the following conditions
\begin{equation*}
\mbox{\emph{ 
$r \neq p$, $r>2$, $r$ divides $|H \cap T|$ and either $c > n/2$, or $c=n/2$ and $T = {\rm P\Omega}_{n}^{-}(q)$.} \label{e:star} \tag{$\star$}}
\end{equation*}
Note that if $r$ divides $|\Omega|$ and all the conditions in \eqref{e:star} hold then $\kappa(T,r)=1$ and thus $T$ is $r$-elusive (see Lemma \ref{l:floor} and Corollary \ref{p:star}).

\renewcommand{\arraystretch}{1.1}
\begin{table}
$$\begin{array}{lll} \hline\hline
\mbox{Case} & \multicolumn{1}{c}{T} & \mbox{Conditions} \\ \hline
(\mathcal{A}1) & \left\{ \begin{array}{ll}
{\rm P\Omega}_{d-1}^{\e}(p) & \mbox{if $(d,p)=1$}\\
{\rm P\Omega}_{d-2}^{\e}(p) & \mbox{otherwise} \end{array} \right. & d \geqs 8, \, p \ne 2 \\

(\mathcal{A}2) & \hspace{4.8mm} {\rm Sp}_{d-2}(2) & d \geqs 10,\, d \equiv 2 \imod{4},\, p = 2 \\

(\mathcal{A}3) & \left\{ \begin{array}{ll} \Omega_{d-2}^{+}(2) & \mbox{if $d \equiv 0 \imod{8}$}\\ \Omega_{d-2}^{-}(2) & \mbox{if $d \equiv 4 \imod{8}$} \end{array} \right. & d \geqs 12,\, d \equiv 0 \imod{4},\, p=2 \\

(\mathcal{A}4) & \left\{ \begin{array}{ll} \Omega_{d-1}^{+}(2) & \mbox{if $d \equiv \pm 1 \imod{8}$}\\ 
\Omega_{d-1}^{-}(2) & \mbox{if $d \equiv \pm 3 \imod{8}$} 
\end{array}\right. & d \geqs 9,\, \mbox{$d$ odd}, \, p=2 \\ \hline \hline
\end{array}$$
\caption{The collection $\mathcal{A}$, $S=A_{d}$}
\label{atab}
\end{table}
\renewcommand{\arraystretch}{1}

\renewcommand{\arraystretch}{1.1}
\begin{table} 
$$\begin{array}{llll} \hline\hline
\mbox{Case} & \hspace{1.6mm} T & \hspace{1.6mm} S & \mbox{Conditions} \\ \hline

(\mathcal{B}1)  & \hspace{1.6mm} {\rm PSp}_{10}(p) &  \hspace{1.6mm} {\rm PSU}_{5}(2) &  p \ne 2 \\ 

(\mathcal{B}2) &  \hspace{1.6mm} {\rm P\Omega}_{8}^{+}(q) & \hspace{-2.2mm}\left\{ \begin{array}{ll}
\Omega_{7}(q) & p>2\\
{\rm Sp}_{6}(q) & p=2 
\end{array} \right. &  \\

(\mathcal{B}3) & \hspace{1.6mm} {\rm P\Omega}_{8}^{+}(q)  & \hspace{1.6mm} 
{}^{3}D_{4}(q_{0}) & q=q_{0}^{3} \\

(\mathcal{B}4)  & \hspace{1.6mm} {\rm P\Omega}_{8}^{+}(p) &  \hspace{1.6mm} \Omega_{8}^{+}(2) & p \ne 2 \\

(\mathcal{B}5)  & \hspace{1.6mm} {\rm PSL}_{7}^{\e}(p) & \hspace{1.6mm} {\rm PSU}_{3}(3) & p\equiv \e\imod{3}, \, p \geqs 5 \\

(\mathcal{B}6) & \hspace{-3mm} \left\{ \begin{array}{ll}
\Omega_{7}(q) & \hspace{-1.5mm} p>2\\
{\rm Sp}_{6}(q) & \hspace{-1.5mm} p=2 
\end{array} \right. & \hspace{1.6mm} G_{2}(q) & q>2,\, V=L(\l_1) \\

(\mathcal{B}7)  &  \hspace{1.6mm} \Omega_{7}(q) & \hspace{1.6mm} G_{2}(q) & p=3, \, V=L(\l_{2}) \\ 

(\mathcal{B}8)  & \hspace{1.6mm} \Omega_{7}(p) &  \hspace{1.6mm} {\rm Sp}_{6}(2) & p \ne 2 \\

(\mathcal{B}9) &  \hspace{1.6mm} {\rm PSL}_{6}^{\e}(q) & \hspace{1.6mm} {\rm PSL}_{3}^{\e}(q) &  p \neq 2, \, V = L(2\l_1) \\

(\mathcal{B}10) &  \hspace{1.6mm} {\rm PSL}_{6}^{\e}(q) & \hspace{1.6mm} A_{7} & q \leqs p^2, \, p\equiv \e\imod{3}, \, p \geqs 5 \\

(\mathcal{B}11)  &  \hspace{1.6mm} {\rm PSL}_{6}^{\e}(q)  &  \hspace{1.6mm} A_{6} & q \leqs p^2, \, p\equiv \e\imod{3}, \, p \geqs 5 \\

(\mathcal{B}12)  & \hspace{1.6mm} {\rm PSL}_{6}^{\e}(p) & \hspace{1.6mm} {\rm PSL}_{3}(4) & p\equiv \e\imod{3}, \, p \geqs 5 \\

(\mathcal{B}13)  & \hspace{1.6mm} {\rm PSL}_{6}^{\e}(p) &  \hspace{1.6mm} {\rm PSU}_{4}(3) &  p\equiv \e\imod{3}, \, p \geqs 5 \\

(\mathcal{B}14)  & \hspace{1.6mm} {\rm PSL}_{6}(3) & \hspace{1.6mm} {\rm M}_{12} & \\

(\mathcal{B}15)  & \hspace{1.6mm} {\rm PSU}_{6}(2) &  \hspace{1.6mm} {\rm PSU}_{4}(3) & \\

(\mathcal{B}16)  & \hspace{1.6mm} {\rm PSU}_{6}(2) & \hspace{1.6mm} {\rm M}_{22} & \\

(\mathcal{B}17)  & \hspace{1.6mm} {\rm PSp}_{6}(q) & \hspace{1.6mm} {\rm J}_{2} & q \leqs p^2, \, p \geqs 3  \\

(\mathcal{B}18)  & \hspace{1.6mm} {\rm PSp}_{6}(p) &  \hspace{1.6mm} {\rm PSU}_{3}(3) & p \neq 3 \\ \hline\hline
\end{array}$$
\caption{The collection $\mathcal{B}$}
\label{btab}
\end{table}
\renewcommand{\arraystretch}{1}

\begin{theorem}\label{t:main}
Let $G \leqs {\rm Sym}(\Omega)$ be a primitive almost simple classical group with socle $T$ and point stabiliser $H \in \mathcal{S}$. Let $S$ denote the socle of $H$ and let $n$ be the dimension of the natural $T$-module. Let $r$ be a prime divisor of $|\Omega|$. Then $T$ is $r$-elusive if and only if one of the following holds:
\begin{itemize}\addtolength{\itemsep}{0.2\baselineskip}
\item[{\rm (i)}] $n<6$ and $(T,S,r)$ is one of the cases recorded in Table \ref{t:lowdim};
\item[{\rm (ii)}] $n \geqs 6$, $H \in \mathcal{A}$ and one of the following holds:

\vspace{1mm}

\begin{itemize}\addtolength{\itemsep}{0.2\baselineskip}
\item[{\rm (a)}] $r=2$, $p \neq 2$, $T = \Omega_n(p)$ and $\left(\frac{(n+1)/2}{p}\right)=1$;
\item[{\rm (b)}] $r=2$, $p \neq 2$, $T = {\rm P\Omega}_{n}^{\e}(p)$, $n \equiv 2\imod{4}$ and $p \equiv 5\e \imod{8}$;
\item[{\rm (c)}] $r \ne p$, $r>2$, $r$ divides $|H \cap T|$ and $c=r-1$;
\end{itemize}  
\item[{\rm (iii)}] $n \geqs 6$, $H \in \mathcal{B}$ and $(T,S,r)$ is one of the cases recorded in Table \ref{t:bex};
\item[{\rm (iv)}] $n \geqs 6$, $H \not\in \mathcal{A} \cup \mathcal{B}$ and all the conditions in \eqref{e:star} hold. 
\end{itemize}
\end{theorem}

\begin{remk}\label{r:main}
As previously remarked, the $r$-elusive $\mathcal{S}$-actions with $n<6$ are determined in \cite[Proposition 6.3.1]{BG}. The relevant cases are listed in Table \ref{t:lowdim}, where the final column records necessary and sufficient conditions for the $r$-elusivity of $T$ (in particular, the given conditions ensure that $r$ divides $|\Omega|$). These are additional to the conditions needed for the existence and maximality of $H$ in $G$, which can be read off from the relevant tables in \cite[Section 8.2]{BHR}, or from \cite[Table 6.3.1]{BG}.  Similarly, we refer the reader to Remark \ref{r:42} for further information on the conditions recorded in the final column of Table \ref{t:bex}.
\end{remk}

\begin{remk}\label{r:main2} 
Note that $r^2$ must divide $q^c-1$ if $(T,S,r)$ is an example arising in part (iv) of Theorem \ref{t:main}. It is easy to see that there are genuine examples. For example, take $T = {\rm P\Omega}_{12}^{+}(p)$, $S = {\rm PSL}_{2}(11)$ and $r=11$, where $p$ is a prime such that $p \equiv -1 \imod{605}$, so $c=10$ and \cite[Table 8.83]{BHR} indicates that $S$ is a maximal subgroup of $T$. Note that there are infinitely many primes of this form by Dirichlet's theorem. 
\end{remk}

\renewcommand{\arraystretch}{1.1}
\begin{table}[h]
$$\begin{array}{llcl} \hline\hline
T & S & r & \mbox{Conditions} \\ \hline
{\rm PSL}_{5}^{\e}(q) & {\rm PSU}_{4}(2) & 2 & \\
& & 5 & q^2 \equiv -1 \imod{25} \\
& {\rm PSL}_{2}(11) & 5 & q^2 \equiv -1 \imod{25} \\
& & 11 & q \not\equiv \e \imod{11},\, q^5 \equiv \e \imod{121} \\
& {\rm M}_{11} & 11 & (\e,q)=(+,3) \\ 
{\rm PSL}_{4}^{\e}(q) & {\rm PSU}_{4}(2) & 2 & q \not\equiv \e \imod{8} \\
& & 3 & q \equiv \e \imod{9} \\
& & 5 & q^2 \equiv -1 \imod{25} \\
& A_7 & 2 & q \equiv 5\e \imod{8} \\
& & 3 & q \equiv -\e \imod{9} \\
& & 5 & q^2 \equiv -1 \imod{25} \\ 
& & 7 & q(q+\e) \equiv -1 \imod{49} \\
& {\rm PSL}_{2}(7) & 2 & q \equiv 5\e \imod{8} \\
& & 7 & q(q+\e) \equiv -1 \imod{49} \\
& {\rm PSL}_{3}(4) & 2 & (\e,q) = (-,3) \\ 
{\rm PSp}_{4}(q)' & A_6 & 2 & q \equiv \pm 1 \imod{12} \\
& & 3 & q^2 \equiv 1 \imod{9} \\
& & 5 &  q^2 \equiv -1 \imod{25} \\ 
& A_7 & 5 & q=7 \\
{\rm PSL}_{3}^{\e}(q) & {\rm PSL}_{2}(7) & 2 & \\
& & 3 & q \equiv 4\e,7\e,8\e \imod{9} \\
& & 7 & \mbox{$q \equiv -\e \imod{49}$ or $q(q+\e) \equiv -1 \imod{49}$} \\
& A_6 & 2 & (\e,q) \ne (-,5) \\
& & 5 & q \equiv -\e \imod{25} \\
& A_7 & 2 & (\e,q) = (-,5) \\
{\rm PSL}_{2}(q) & A_5 & 2 & q \equiv \pm 1 \imod{8} \\
& & 3,5 & q \equiv \pm 1 \imod{r^2} \\ 
\hline\hline
\end{array}$$
\caption{The $r$-elusive $\mathcal{S}$-actions, $n<6$}
\label{t:lowdim}
\end{table}

\renewcommand{\arraystretch}{1.1}
\begin{table}
$$\begin{array}{lllcl}\hline\hline
\mbox{Case} & \hspace{1.6mm} T & \hspace{2.4mm} S &  r & \mbox{Conditions} \\ \hline

(\mathcal{B}1)  & \hspace{1.6mm} {\rm PSp}_{10}(p) & \hspace{2.4mm} {\rm PSU}_{5}(2) & 2 & p \equiv \pm 1 \imod{8} \\
& & & 11 & p^2 \not\equiv 1 \imod{11},\, p^5 \equiv \pm 1 \imod{121} \\ 

(\mathcal{B}4)  & \hspace{1.6mm} {\rm P\Omega}_{8}^{+}(p) & \hspace{2.4mm} \Omega_{8}^{+}(2) &  2 & p \geqs 7 \\
& & & 3 & p^2 \equiv 1 \imod{9} \\
& & & 5 & p^2 \equiv -1 \imod{25} \\
& & & 7 & p^2 \not\equiv 1 \imod{7},\, p^3 \equiv \pm 1 \imod{49} \\

(\mathcal{B}5)  & \hspace{1.6mm} {\rm PSL}_{7}^{\e}(p) & \hspace{2.4mm} {\rm PSU}_{3}(3) & 7 & p^2 \not\equiv 1 \imod{7},\,  p^3 \equiv -\e \imod{49} \\

(\mathcal{B}8)  & \hspace{1.6mm} \Omega_{7}(p) & \hspace{2.4mm} {\rm Sp}_{6}(2) & 2 & p \geqs 7 \\
& & & 3 & p^2 \equiv 1 \imod{9} \\
& & & 5 & p^2 \equiv -1 \imod{25} \\
& & & 7 & p^2 \not\equiv 1 \imod{7},\, p^3 \equiv \pm 1 \imod{49} \\

(\mathcal{B}10) &  \hspace{1.6mm} {\rm PSL}_{6}^{\e}(q) & \hspace{2.4mm} A_{7} & 5 & q^2 \equiv -1 \imod{25} \\
& & & 7 & q^2 \not\equiv 1 \imod{7},\, q^3 \equiv -\e \imod{49} \\

(\mathcal{B}11)  &  \hspace{1.6mm} {\rm PSL}_{6}^{\e}(q)  &  \hspace{2.4mm} A_{6} & 5 & q^2 \equiv -1 \imod{25} \\

(\mathcal{B}12)  & \hspace{1.6mm} {\rm PSL}_{6}^{\e}(p) & \hspace{2.4mm} {\rm PSL}_{3}(4) & 2 & p \equiv -5\e \imod{24} \\
& & & 5 & p^2 \equiv -1 \imod{25} \\
& & & 7 & p^2 \not\equiv 1 \imod{7},\, p^3 \equiv -\e \imod{49} \\

(\mathcal{B}13)  & \hspace{1.6mm} {\rm PSL}_{6}^{\e}(p) &  \hspace{2.4mm} {\rm PSU}_{4}(3) &  2 & p\equiv \e\imod{12} \\
& & & 5 & p^2 \equiv -1 \imod{25} \\
& & & 7 & p^2 \not\equiv 1 \imod{7},\, p^3 \equiv -\e \imod{49} \\

(\mathcal{B}14)  & \hspace{1.6mm} {\rm PSL}_{6}(3) & \hspace{2.4mm} {\rm M}_{12} & 2,11 & \\

(\mathcal{B}15)  & \hspace{1.6mm} {\rm PSU}_{6}(2) & \hspace{2.4mm} {\rm PSU}_{4}(3) & 2 & \\

(\mathcal{B}17)  & \hspace{1.6mm} {\rm PSp}_{6}(q) & \hspace{2.4mm} {\rm J}_{2} & 2 & q^2 \equiv 1 \imod{8} \\
& & & 5 & q^2 \equiv -1 \imod{125} \\
& & & 7 & q^2 \not\equiv 1 \imod{7},\, q^3 \equiv \pm 1 \imod{49}  \\ 

(\mathcal{B}18)  & \hspace{1.6mm} {\rm PSp}_{6}(p) & \hspace{2.4mm} {\rm PSU}_{3}(3) & 2 & p \equiv \pm 1 \imod{12} \\
& & & 7 & p^2 \not\equiv 1 \imod{7},\, p^3 \equiv \pm 1 \imod{49} \\ \hline\hline
\end{array}$$
\caption{The $r$-elusive $\mathcal{S}$-actions, $H \in \BB$}
\label{t:bex}
\end{table}
\renewcommand{\arraystretch}{1}

\begin{coroll}\label{c:main11}
Let $G \leqs {\rm Sym}(\Omega)$ be a primitive almost simple classical group with socle $T$ and point stabiliser $H \in \mathcal{S}$. Let $S$ denote the socle of $H$ and let $n$ be the dimension of the natural $T$-module. Let $r$ be a prime dividing $|\Omega|$ and $|H \cap T|$, and define $\kappa(T,r)$ as in \eqref{e:kop}. Then $T$ is $r$-elusive if and only if one of the following holds:
\begin{itemize}\addtolength{\itemsep}{0.2\baselineskip}
\item[{\rm (i)}] $\kappa(T,r)=1$;
\item[{\rm (ii)}] $r \geqs 5$, $r \ne p$, $H \in \mathcal{A}$ and $c=r-1$; 
\item[{\rm (iii)}] $r \in \{2,3\}$ and $(T,S,r)$ is one of the cases recorded in Table \ref{tab:11}.
\end{itemize}
In particular, if $n>10$ then $T$ is $r$-elusive only if $\kappa(T,r)=1$ or $H \in \mathcal{A}$.
\end{coroll}

\renewcommand{\arraystretch}{1.1}
\begin{table}
$$\begin{array}{llcl}\hline\hline
T & S & r & \mbox{Conditions} \\ \hline
\Omega_{d-\delta}(p) & A_{d} & 2 & \mbox{$p \ne 2$, $d-\delta$ odd, $\left(\frac{(d-\delta+1)/2}{p}\right)=1$}  \\
{\rm P\Omega}_{d-\delta}^{\e}(p) & A_{d} & 2 & \mbox{$p \ne 2$, $(d-\delta) \equiv 2 \imod{4}$, $p \equiv 5\e \imod{8}$} \\
{\rm PSp}_{10}(p) & {\rm PSU}_{5}(2) & 2 & p \equiv \pm 1 \imod{8} \\
{\rm P\Omega}_{8}^{+}(p) & \Omega_{8}^{+}(2) &  2 & p \geqs 7 \\
& & 3 & p^2 \equiv 1 \imod{9} \\
\Omega_{7}(p) & {\rm Sp}_{6}(2) & 2 & p \geqs 7 \\
& & 3 & p^2 \equiv 1 \imod{9} \\
{\rm PSL}_{6}^{\e}(q) & {\rm PSL}_{3}(4) & 2 & p \equiv -5\e \imod{24} \\
&  {\rm PSU}_{4}(3) &  2 & p\equiv \e\imod{12} \\
& {\rm M}_{12} & 2 & (\e,q) = (+,3) \\
{\rm PSU}_{6}(2) & {\rm PSU}_{4}(3) & 2 & \\
{\rm PSp}_{6}(q) & {\rm J}_{2} & 2 & q^2 \equiv 1 \imod{8} \\
& {\rm PSU}_{3}(3) & 2 & q=p \equiv \pm 1 \imod{12} \\
{\rm PSL}_{5}^{\e}(q) & {\rm PSU}_{4}(2) & 2 & \\
{\rm PSL}_{4}^{\e}(q) & {\rm PSU}_{4}(2) & 2 & q \equiv -\e \imod{4} \\
& & 3 & q \equiv \e \imod{9} \\
{\rm PSL}_{4}^{\e}(q) & A_7 & 3 & q \equiv -\e \imod{9} \\
{\rm PSp}_{4}(q)' & A_6 & 2 & q \equiv \pm 1 \imod{12} \\
& & 3 & q^2 \equiv 1 \imod{9} \\ \hline\hline
\multicolumn{4}{l}{\mbox{{\small $\delta=2$ if $p$ divides $d$, otherwise $\delta=1$}}}
\end{array}$$
\caption{The $r$-elusive $\mathcal{S}$-actions with $r \in \{2,3\}$ and $\kappa(T,r) \geqs 2$}
\label{tab:11}
\end{table}
\renewcommand{\arraystretch}{1}

This corollary is easily deduced from Theorem \ref{t:main}. To see this, first recall that 
$\kappa(T,r)=1$ if the conditions in \eqref{e:star} hold, and one checks that the same is true for all the cases in Tables \ref{t:lowdim} and \ref{t:bex} with $r \geqs 5$. To determine the examples appearing in Table \ref{tab:11} we include the two cases with $H \in \mathcal{A}$ and $r=2$ in Theorem \ref{t:main}(ii), and we exclude those in Tables \ref{t:lowdim} and \ref{t:bex} with $r \in \{2,3\}$ and $\kappa(T,r)=1$.

Notice that there are genuine examples arising in parts (i) and (ii) of Corollary \ref{c:main11}. For (i), see Remark \ref{r:main2}. For an example in (ii), take $T=\Omega_{15}(p)$, $S=A_{16}$ and $r=7$, where $p$ is a prime such that $p(p-1) \equiv -1 \imod{49}$ (here $T$ is $r$-elusive and $\kappa(T,r) = 2$).

\begin{remk}\label{r:main3}
We can immediately determine the $r$-elusive $\mathcal{S}$-actions with $r=2$ or $3$ from Theorem \ref{t:main} (there are no examples if $H \not\in \mathcal{A}$ and $n>10$). It is also worth noting that the only $p$-elusive $\mathcal{S}$-action, where $p$ is the defining characteristic, is the case labelled $(\mathcal{B}15)$ in Table \ref{btab} with $T = {\rm PSU}_{6}(2)$ and $S = {\rm PSU}_{4}(3)$.
\end{remk}

Let us briefly describe the proof of Theorem \ref{t:main}. First, recall that we may assume $n \geqs 6$, where $n$ is the dimension of the natural module for $T$. We start by considering the collections $\mathcal{A}$ and $\mathcal{B}$, which are handled directly in Sections \ref{s:a} and \ref{s:b}, respectively. In order to complete the proof, we may assume that $H \not\in \mathcal{A} \cup \mathcal{B}$ and it remains to show that $T$ is $r$-elusive if and only if all the conditions in \eqref{e:star} hold.

To do this, we first apply a key theorem of Guralnick and Saxl (see Theorem \ref{gursax}), which provides an important reduction to the situation where
$$\mbox{$r \neq p$, $r>2$, $r$ divides $|H\cap T|$ and $c > \max\{2,\sqrt{n}/2\}$}$$
with $c$ the integer defined in \eqref{e:eqcc}. This is the content of Proposition \ref{p:thm1} and it essentially reduces the problem to showing that $T$ contains a derangement of order $r$ whenever $r \ne p$ is an odd prime such that 
$$\max\{2,\sqrt{n}/2\} < c \leqs n/2$$
(see Proposition \ref{p:c} for the precise statement).

Not surprisingly, most of the work in this final part of the argument arises when $S$ is a simple group of Lie type. Here the analysis naturally splits into two cases, according to whether or not $S \in {\rm Lie}(p)$, where ${\rm Lie}(p)$ is the set of simple groups of Lie type in the defining characteristic $p$ (the case $S \not\in {\rm Lie}(p)$ is studied in Section \ref{ss:lie_cross} and the proof is completed in Section \ref{ss:lie_def}). A similar approach applies in both cases; either we identify a specific derangement of order $r$ (this is often an element $x \in T$ of order $r$ with the largest possible $1$-eigenspace on the natural module), or we argue by estimating, and then comparing, the number of conjugacy classes of elements (or subgroups) of order $r$ in $T$ and $H \cap T$, respectively. For $S \not\in {\rm Lie}(p)$, the analysis relies heavily on the well known bounds of Landazuri and Seitz \cite{Land-S} on the dimensions of irreducible representations. In the defining characteristic, we use the highest weight theory of irreducible representations of quasisimple groups and the ambient simple algebraic groups. Work of Hiss and Malle \cite{HM} and L\"{u}beck \cite{Lu} also plays an important role.

\vs

We conclude by presenting several corollaries that are obtained by combining Theorem \ref{t:main} with the main results of \cite{BG} on geometric actions of classical groups. We follow \cite{KL} in labelling the geometric subgroup collections $\C_1, \ldots, \C_8$ (see \cite[Table 1.4.2]{BG}).

\begin{coroll}\label{c:main2}
Let $G \leqs {\rm Sym}(\Omega)$ be a primitive almost simple classical group over $\mathbb{F}_{q}$, where $q=p^a$ with $p$ prime. Let $T$ and $H$ be the socle and point stabiliser of $G$, and let $r$ be a prime divisor of $|\Omega|$. 
\begin{itemize}\addtolength{\itemsep}{0.2\baselineskip}
\item[{\rm (i)}] If $H \in \C_1 \cup \C_2$, $r=p>2$ and $T$ is $r$-elusive, then $(G,H)$ belongs to a known list of cases.
\item[{\rm (ii)}] If $H \in \mathcal{S}$ then $T$ is $r$-elusive if and only if $(G,H,r)$ satisfies the conditions in Theorem \ref{t:main}.
\item[{\rm (iii)}] In all other cases, $T$ is $r$-elusive if and only if $(G,H,r)$ belongs to a known list of cases.
\end{itemize}
\end{coroll}

\begin{remk}\label{r:main4}
In part (i) of Corollary \ref{c:main2}, we refer the reader to \cite[Theorems 4.1.4 and 5.1.2]{BG} for further details. Similarly in (iii), if $H \in \C_i$ then the relevant cases are recorded in \cite[Theorem 5.i.1]{BG}.  It is important to note that we are not claiming to have a complete classification of all the $r$-elusive classical groups. Indeed, in \cite{BG} we are unable to determine necessary and sufficient conditions for $p$-elusivity in all cases when $p$ is odd and $H \in \C_1 \cup \C_2$.
\end{remk}

Next we extend \cite[Theorem 1.5.3]{BG} to give a complete description of the $2$-elusive  almost simple primitive classical groups (note that $\kappa(T,2) \geqs 2$ if $n \geqs 6$, where $n$ is the dimension of the natural module for $T$).

\begin{coroll}\label{c:main3}
Let $G \leqs {\rm Sym}(\Omega)$ be a primitive almost simple classical group over $\mathbb{F}_{q}$ with socle $T$ and point stabiliser $H$. Then $T$ is $2$-elusive if and only if $|\Omega|$ is even and one of the following holds:
\begin{itemize}\addtolength{\itemsep}{0.2\baselineskip}
\item[{\rm (i)}] $H \in \C_1$ and $(G,H)$ is one of the cases in \cite[Table 4.1.3]{BG};
\item[{\rm (ii)}] $H \in \mathcal{S}$ and $(G,H)$ is one of the cases in Tables \ref{t:lowdim} or \ref{tab:11} (with $r=2$); 
\item[{\rm (iii)}] $H \not\in \C_1 \cup \mathcal{S}$ and $(G,H)$ is one of the cases in \cite[Table 5.1.2]{BG}.
\end{itemize}
\end{coroll}

Finally, by combining Corollary \ref{c:main11} with the main results in \cite[Sections 5.3-5.9]{BG} we obtain Corollary \ref{c:main4} below. Notice that we omit the $\C_i$-actions with $i \in \{1,2\}$ since it is not possible to state a definitive result in these cases when $r=p$ (and the required conditions when $r \ne p$ are rather complicated); we refer the reader to the relevant discussion in \cite[Chapter 4 and Section 5.1]{BG}). In order to state the result, let us say that a partition $\l = (\l_1, \ldots, \l_t)$ of a positive integer is \emph{$p$-bounded} (with $p$ prime) if $\l_i \leqs p$ for all $i$. In addition, set $d(\l) = \gcd\{\l_1, \ldots, \l_t\}$.  

\begin{coroll}\label{c:main4}
Let $G \leqs {\rm Sym}(\Omega)$ be a primitive almost simple classical group over $\mathbb{F}_{q}$, where $q=p^a$ with $p$ prime. Let $T$ and $H$ be the socle and point stabiliser of $G$, respectively, and let $n$ denote the dimension of the natural $T$-module. Let $r>5$ be a prime divisor of $|\Omega|$ and assume $n>5$ and $H \not\in \C_1 \cup \C_2$. Define $\kappa(T,r)$ as in \eqref{e:kop}. Then $T$ is $r$-elusive if and only if one of the following holds:
\begin{itemize}\addtolength{\itemsep}{0.2\baselineskip}
\item[{\rm (i)}] $\kappa(T,r)=1$;
\item[{\rm (ii)}] $r \ne p$, $H \in \mathcal{A}$ and $c=r-1$;
\item[{\rm (iii)}] $H \in \C_5$ is a subfield subgroup over $\mathbb{F}_{q_0}$ with $q=q_0^k$ for an odd prime $k$ and either $r=k$, or $r=p$ and the following conditions hold if $T= {\rm PSL}_{n}^{\e}(q)$:
\begin{itemize}\addtolength{\itemsep}{0.2\baselineskip}
\item[{\rm (a)}] $(d(\l),q-\e) = (d(\l),q_0-\e)$ for every $p$-bounded partition $\l$ of $n$; 
\item[{\rm (b)}] If there is a partition $\l$ in (a) with $(d(\l),q-\e)>1$, then either $k \geqs p$ or $(k,(n,q_0-\e))=1$.
\end{itemize}
\item[{\rm (iv)}] $r=p$, $T={\rm PSL}_{n}(q)$, $H \in \C_8$ is of type ${\rm GU}_{n}(q_0)$ with $q=q_0^2$, $n$ is odd and $(d(\l),q-1) = (d(\l),q_0+1)$ for every $p$-bounded partition $\l$ of $n$. 
\end{itemize}
\end{coroll}

See \cite[Sections 5.5 and 5.8]{BG} for an explanation of the number-theoretic conditions appearing in parts (iii) and (iv) of Corollary \ref{c:main4}. Notice that there are genuine $p$-elusive examples in these cases. For instance, in (iii) we see that $T = {\rm PSL}_{6}(q)$ is $p$-elusive when $p \geqs 7$ and $k \geqs 5$. Similarly in (iv), if we take $p \geqs 7$ then $T = {\rm PSL}_{7}(q)$ is $p$-elusive.

\vs

\noindent \textbf{Notation.} We adopt the notation of \cite{BG, KL} for classical groups, so for example we write ${\rm PSL}_{n}^{+}(q) = {\rm PSL}_{n}(q)$ and ${\rm PSL}_{n}^{-}(q) = {\rm PSU}_{n}(q)$. We also use the standard notation for labelling involution class representatives presented in \cite{GLS} and \cite{AS}, in the odd and even characteristic settings, respectively. We use the notation in \cite{BG} for representatives of conjugacy classes of elements of odd prime order, which is recalled in Section \ref{ss:cc}. Finally, if $n$ is a positive integer then $Z_n$ (or just $n$) denotes a cyclic group of order $n$.

\vs

\noindent \textbf{Acknowledgments.} We thank an anonymous referee for helpful comments on an earlier version of the paper. The second author is supported by the Australian Research Council Grant DP160102323.

\section{Preliminaries}\label{s:prel}

In this section we record some preliminary results which will be needed in the proof of Theorem \ref{t:main}. 

\subsection{Derangements}

We begin with a useful lemma on derangements in the socle of a primitive almost simple group.

\begin{lem}\label{l:count0}
Let $G \leqs {\rm Sym}(\Omega)$ be an almost simple primitive group with socle $T$ and point stabiliser $H$. Set $H_0 = H \cap T$ and let $\Omega_0$ be the set of right cosets of $H_0$ in $T$. Then $\Delta(T) = \Delta_0(T)$, where $\Delta(T)$ and $\Delta_0(T)$ denote the set of derangements in $T$ on $\Omega$ and $\Omega_0$, respectively. In particular, if $r$ is a prime divisor of $|\Omega|$ then $T$ is $r$-elusive on $\Omega$ if and only if $T$ is $r$-elusive on $\Omega_0$. 
\end{lem}

\begin{proof}
First observe that $|\Omega|=|\Omega_0|$. Suppose $x \in \Delta(T)$. If $x$ has a fixed point on $\Omega_0$ then $x \in H_0^t$ for some $t \in T$, so $x \in H^t$ and thus $x$ fixes a point of $\Omega$, which is a contradiction. Therefore, $\Delta(T) \subseteq \Delta_0(T)$. Now assume $y \in \Delta_0(T)$ and suppose $y$ fixes a point of $\Delta$, so $y \in H^g \cap T$ for some $g \in G$. Since $G=HT$, we can write $g=ht$ for some $h \in H$, $t \in T$, so $y \in H^t \cap T = H_0^t$, but this contradicts the fact that $y$ is a derangement on $\Omega_0$. The result follows.
\end{proof}

\begin{cor}\label{c:count}
Let $G \leqs {\rm Sym}(\Omega)$ be an almost simple primitive group with socle $T$ and point stabiliser $H$. Let $r$ be a prime divisor of $|\Omega|$ and set $H_0 = H \cap T$. Suppose there are more $T$-classes of elements (or subgroups) of order $r$ in $T$ than there are $H_0$-classes of such elements (or subgroups) in $H_0$. Then $T$ is not $r$-elusive on $\Omega$. 
\end{cor}

\subsection{Conjugacy classes}\label{ss:cc}

The conjugacy classes of elements of prime order in the almost simple classical groups are studied in \cite[Chapter 3]{BG}, which brings together earlier work of Wall \cite{Wall}, Aschbacher and Seitz \cite{AS}, Liebeck and Seitz \cite{LS}, Gorenstein, Lyons and Solomon \cite{GLS} and others. In order to highlight some of the results and the relevant notation, let us focus on conjugacy in the general linear group $G = {\rm GL}_{n}(q)$, where $q=p^f$ with $p$ a prime. Let $V$ be the natural module.

Let $x \in G$ be an element of prime order $r$. If $r \ne p$ then $x$ is diagonalisable over 
$\mathbb{F}_{q^i}$, but not over any proper subfield, where $i=\Phi(r,q)$ is the integer
\begin{equation}\label{e:phirq}
\Phi(r,q) = \min\{i \in \mathbb{N}\,:\, \mbox{$r$ divides $q^i-1$}\}.
\end{equation}
In other words, $r$ is a primitive prime divisor of $q^i-1$. By Maschke's Theorem, $x$ fixes a direct sum decomposition
$$V = U_1 \oplus \cdots \oplus U_m \oplus C_V(x),$$
where each $U_j$ is an $i$-dimensional subspace on which $x$ acts irreducibly, and $C_V(x)$ denotes the $1$-eigenspace of $x$. The eigenvalues of $x$ on $U_j \otimes \mathbb{F}_{q^i}$ are of the form $\L = \{\l, \l^q, \ldots, \l^{q^{i-1}}\}$ for some nontrivial $r$-th root of unity $\l \in \mathbb{F}_{q^i}$. In total, there are $t=(r-1)/i$ possibilities for $\L$, say $\L_1, \ldots, \L_t$ (these are simply the orbits on the set of nontrivial $r$-th roots of unity in $\mathbb{F}_{q^i}$ under the permutation $\omega \mapsto \omega^q$). Following \cite{BG}, if $a_j$ denotes the multiplicity of $\L_j$ in the multiset of eigenvalues of $x$ on $V \otimes \mathbb{F}_{q^i}$, then we will write
$$x = [\L_1^{a_1}, \ldots, \L_t^{a_t},I_{e}],$$
where $e=\dim C_V(x)$. This convenient notation is justified by \cite[Lemma 3.1.7]{BG}, which states that two elements of order $r$ in $G$ are conjugate if and only if they have the same multiset of eigenvalues (in $\mathbb{F}_{q^i}$). 

There is a similar description of the semisimple conjugacy classes of elements of prime order in the other classical groups, with some suitable modifications. For instance, if $x \in {\rm Sp}_{n}(q)$ and $ir$ is odd, then $t=(r-1)/i=2s$ is even and the $\L_j$ can be labelled so that $\L_{j}^{-1} = \L_{s+j}$ for $1 \leqs j \leqs s$ (where $\L_j^{-1} = \{\l^{-1} \,:\, \l \in \L_j\}$). Then the fact that $x$ preserves a symplectic form on $V$ implies that $a_{j}=a_{s+j}$ for each $j$, so we can write
$$x = [(\L_1,\L_1^{-1})^{a_1}, \ldots, (\L_s,\L_s^{-1})^{a_s},I_{e}].$$
Once again, two elements of order $r$ are conjugate if and only if they have the same eigenvalues. We refer the reader to \cite[Chapter 3]{BG} for further details.

\begin{rem}\label{r:cc}
Let $T$ be a simple classical group over $\mathbb{F}_{q}$ with natural module $V$ and let 
$x \in T$ be an element of odd prime order $r \ne p$. Set $n = \dim V$, $i = \Phi(r,q)$ and assume $c \geqs 2$, where $c$ is the integer in \eqref{e:eqcc}. By \cite[Lemma 3.1.3]{BG} we may write $x = \hat{x}Z$, where $\hat{x} \in {\rm GL}(V)$, $Z = Z({\rm GL}(V))$ and $\hat{x}$ has order $r$. Here $\hat{x}$ is conjugate to a block-diagonal matrix of the form 
$[X_1^{a_1}, \ldots, X_{s}^{a_s},I_{e}]$,
where $s=(r-1)/c$ and the $X_j$ are distinct $c \times c$ matrices with distinct eigenvalues in $\mathbb{F}_{q^i}$ (here $a_j$ denotes the multiplicity of $X_j$ as a diagonal block of $\hat{x}$). For example, if $T = {\rm PSL}_{n}(q)$ then $c=i$ and $X_j$ is irreducible with eigenvalues $\L_j$ as above. In particular, there exists an element $x \in T$ of order $r$ such that $\dim C_V(\hat{x}) = n-c$ (and the nontrivial eigenvalues of such an element (in 
$\mathbb{F}_{q^i}$) are distinct).
\end{rem}

Now suppose $x \in G$ has order $r=p$. Here we can write
\begin{equation}\label{e:unip}
x = [J_p^{a_p}, J_{p-1}^{a_{p-1}}, \ldots, J_1^{a_1}],
\end{equation}
where $J_i$ is a standard unipotent Jordan block of size $i$, and $a_i$ denotes the multiplicity of $J_i$ in the Jordan form of $x$ on $V$. In ${\rm GL}_{n}(q)$, two elements of order $p$ are conjugate if and only if they have the same Jordan form. There is a similar description of the conjugacy classes of elements of order $p$ in the other classical groups (again, we refer the reader to \cite[Chapter 3]{BG}).  

In the proof of Theorem \ref{t:main}, we will often establish the existence of a derangement of order $r$ by comparing the number of $T$-classes of subgroups (or elements) in $T$ with the number of such $H_0$-classes in $H_0$ (recall that if the former is greater than the latter, then $T$ contains a derangement of order $r$ by Corollary \ref{c:count}). Therefore, it will be helpful to have some general bounds on the number of such classes. With this aim in mind, the following notation will be useful.

\vs

\noindent \textbf{Notation.} Let $G$ be a finite group and let $m$ be a positive integer. We write $\kappa(G,m)$ for the number of conjugacy classes of subgroups of order $m$ in $G$. 

\begin{lem}\label{l:floor}
Let $T$ be a simple classical group over $\mathbb{F}_{q}$, let $n$ be the dimension of the natural module and let $r \ne p$ be an odd prime divisor of $|T|$. Set $m= \lfloor n/c \rfloor$, where  $c$ is the integer in \eqref{e:eqcc}. Assume $c \geqs 2$.
\begin{itemize}\addtolength{\itemsep}{0.2\baselineskip}
\item[{\rm (i)}] $\kappa(T,r)=1$ if and only if $m=1$, or $T={\rm P\Omega}_{n}^{-}(q)$ and $c=n/2$.
\item[{\rm (ii)}] If $m =2$, with $c \ne n/2$ if $T = {\rm P\Omega}_{n}^{+}(q)$, then $\kappa(T,r) \leqs (r-1)/c+1$.
\item[{\rm (iii)}] $\kappa(T,r) \geqs m-\delta$, where $\delta=1$ if $T$ is an orthogonal group and $n=mc$, otherwise $\delta=0$. In particular, $\kappa(T,r) \geqs \lfloor n/(r-1)\rfloor-1$.
\end{itemize}
\end{lem}

\begin{proof}
First consider (i). If $m \geqs 2$ (and $c \ne n/2$ if $T={\rm P\Omega}_{n}^{-}(q)$) then $\la [X_1,I_{n-c}]Z\ra$ and $\la [X_1^2,I_{n-2c}]Z\ra$ represent two distinct $T$-classes of subgroups of order $r$, so $\kappa(T,r) \geqs 2$. For the converse, let us assume $m=1$, or $T={\rm P\Omega}_{n}^{-}(q)$ and $c=n/2$. We claim that $\kappa(T,r)=1$.

Let $x \in T$ be an element of order $r$. By replacing $x$ with a suitable conjugate, if necessary, we may assume that $x=\hat{x}Z$ with $\hat{x} = [X_1^{a_1}, \ldots, X_{s}^{a_s},I_{e}]$ as in Remark \ref{r:cc} (so $s=(r-1)/c$). Suppose $T \ne {\rm P\Omega}_{n}^{\pm}(q)$. Since each $X_j$ has size $c$ we have $\hat{x}=[X_j,I_{n-c}]$ for some $j$, hence $T$ has $s$ conjugacy classes of elements of order $r$. Now the eigenvalues of $X_j$ coincide with the eigenvalues of a suitable power of $X_1$, so $\la x \ra$ is $T$-conjugate to $\la y \ra$, where $y = \hat{y}Z \in T$ and $\hat{y} = [X_1,I_{n-c}]$, so $\kappa(T,r)=1$ as claimed. 

Now assume $T = {\rm P\Omega}_{n}^{\e}(q)$ with $\e=\pm$. If $n/2<c<n$ then the above argument goes through unchanged. If $c=n/2$ then $\e=-$ and $C_V(\hat{x})$ has to be nontrivial (see \cite[Remark 3.5.5(iii)]{BG}), so $\hat{x}=[X_j,I_{n/2}]$ and the same argument applies. Finally, suppose $c=n$. There are two cases to consider:
\begin{itemize}\addtolength{\itemsep}{0.2\baselineskip}
\item[{\rm (a)}] $T = {\rm P\Omega}_{n}^{+}(q)$, $n \equiv 2 \imod{4}$ and $r$ is a primitive prime divisor of $q^{n/2}-1$. 
\item[{\rm (b)}] $T = {\rm P\Omega}_{n}^{-}(q)$ and $r$ is a primitive prime divisor of $q^{n}-1$.
\end{itemize}
Here $\hat{x}=[X_j]$ and for each choice of $j$ there are two $T$-classes of elements of this form, which are fused in ${\rm PO}_{n}^{\e}(q)$ (see \cite[Proposition 3.5.8]{BG}). 
In both cases, we observe that a Sylow $r$-subgroup of $\Omega_{n}^{\e}(q)$ is contained in a cyclic maximal torus of ${\rm O}_{n}^{\e}(q)$ of order $q^{n/2}-\e$. In particular, the Sylow $r$-subgroups of $T$ are cyclic and we conclude that $\kappa(T,r)=1$.  

Now let us turn to (ii). As in (i), if $T={\rm P\Omega}_{n}^{-}(q)$ and $c=n/2$ then $T$ has a unique class of subgroups of order $r$, so for the remainder we may assume that $c \ne n/2$ if $T$ is an orthogonal group. Let $\la x \ra$ be a subgroup of $T$ of order $r$. Since 
$m=2$, it is easy to see that $\la x \ra$ is $T$-conjugate to one of 
$\la [X_1,I_{n-c}]Z \ra$ or $\la [X_1,X_{j},I_{n-2c}]Z \ra$ for some $j \in \{1, \ldots, (r-1)/c\}$.
The result follows.

Finally, consider (iii). Clearly, none of the subgroups $\la [X_1^{a},I_{n-ac}]Z \ra$ are $T$-conjugate, where $1 \leqs a < m$. In addition, if either $n>mc$, or $n=mc$ and $T$ is not an orthogonal group, then $\la [X_1^{m},I_{n-mc}]Z \ra$ represents an additional class of subgroups of order $r$.
\end{proof}

\begin{rem}
The definition of $\delta$ in part (iii) of Lemma \ref{l:floor} can be explained as follows. Let 
$T = {\rm P\Omega}_{n}^{\e}(q)$ and set $i = \Phi(r,q)$ as in \eqref{e:phirq}, so $c=2i$ if $i$ is odd, otherwise $c=i$. Suppose $i=r-1$ and $n=mi=mc$. If $\e=(-)^{m-1}$ then $C_V(\hat{x})$ is nontrivial for all $x \in T$ of order $r$ (see \cite[Remark 3.5.5]{BG}), so the subgroups $\la [X_1^{a},I_{n-ac}]Z \ra$ with $1 \leqs a < m$ form a complete set of representatives of the $T$-classes of subgroups of order $r$. 
\end{rem}

\begin{rem}
Observe that the inequality in Lemma \ref{l:floor}(ii) need not be equality since 
$\la [X_1,X_j,I_{n-2c}]Z\ra$ and $\la [X_1,X_k,I_{n-2c}]Z\ra$ may be conjugate for 
$j \neq k$. For example, suppose $T={\rm PSL}_4(16)$ and $r=17$, so $c=m=2$ and $X_i$ has eigenvalues $\{\omega^i,\omega^{r-i}\}$ for some $r$-th root of unity
$\omega$. Then $[X_1,X_2]^8$ and $[X_1,X_8]$ are $T$-conjugate (they have the same set of eigenvalues), so $\la [X_1,X_2]Z \ra$ and $\la [X_1,X_8]Z \ra$ are conjugate subgroups. 
\end{rem}

The next result follows immediately from part (i) of Lemma \ref{l:floor}; in the statement, we refer to the conditions recorded in \eqref{e:star} (see p.2). 

\begin{cor}\label{p:star}
Let $G \leqs {\rm Sym}(\Omega)$ be an almost simple primitive classical group over $\mathbb{F}_{q}$ with socle $T$ and point stabiliser $H$. Let $n$ be the dimension of the natural $T$-module and let $r$ be a prime divisor of $|\Omega|$. If all of the conditions in \eqref{e:star} hold, then $T$ is $r$-elusive. 
\end{cor}

\begin{lem}\label{l:csub}
Let $T_1$ and $T_2$ be finite simple classical groups over $\mathbb{F}_{q}$, where $q=p^{f}$ and $p$ is a prime. Let $n_1$ and $n_2$ be the dimensions of the respective natural modules and let $r \ne p$ be an odd prime divisor of $|T_1|$ and $|T_2|$. Set $i = \Phi(r,q)$ and
$$c_j = \left\{\begin{array}{ll}
2i & \mbox{if $i$ is odd and $T_j \neq {\rm PSL}_{n_j}(q)$} \\
i/2 &  \mbox{if $i \equiv 2\imod{4}$ and $T_j = {\rm PSU}_{n_j}(q)$} \\
i & \mbox{otherwise}
\end{array}\right.$$
and assume that $c_1 \geqs c_2 \geqs 2$ and $n_2>2n_1$.
Then $\kappa(T_2,r)>\kappa(T_1,r)$.
\end{lem}

\begin{proof}
First assume $c_1=c_2=c$ and set $s=(r-1)/c$. Let $\{\la x_j \ra \,:\, 1 \leqs j \leqs \kappa(T_1,r)\}$ be a set of representatives of the $T_1$-classes of subgroups of order $r$. Write $x_j = \hat{x}_jZ$ with 
\begin{equation}\label{e:xj}
\hat{x}_j = [X_1^{a_{1,j}}, \ldots, X_s^{a_{s,j}},I_{e_j}]
\end{equation}
(up to conjugacy). By relabelling, if necessary, we may assume that there is an integer $\ell \geqs 0$ such that $e_j>0$ if and only if $j>\ell$.

Define elements $y_j, z_k \in T_2$ of order $r$ by setting
$$\begin{array}{ll}
\hat{y}_j = [X_1^{a_{1,j}}, \ldots, X_s^{a_{s,j}},I_{e_j+n_2-n_1}] & 1 \leqs j \leqs \kappa(T_1,r) \\
\hat{z}_k = [X_1^{2a_{1,k}}, \ldots, X_s^{2a_{s,k}},I_{n_2-2n_1}] & 1 \leqs k \leqs \ell.
\end{array}$$
Note that the $1$-eigenspaces of $\hat{y}_j$ and $\hat{z}_k$ are nontrivial, so $y_j$ and $z_k$ are indeed elements of $T_2$.
Then none of the following subgroups 
\begin{equation}\label{e:subg}
\{\la y_j \ra, \; \la z_k \ra \,:\, \ell<j \leqs \kappa(T_1,r),\, 1 \leqs k \leqs \ell\}
\end{equation}
are $T_2$-conjugate, so $\kappa(T_2,r) \geqs \kappa(T_1,r)$. The desired result now follows because it is easy to see that $T_2$ has some additional classes of subgroups of order $r$. For example, if we take $x = \hat{x}Z \in T_2$ with  
$$\hat{x} = \left\{\begin{array}{ll}
\mbox{$[X_1^{2\lfloor n_1/c \rfloor-1},I_{n_2 - c(\lfloor n_2/c \rfloor-1)}]$} & \lfloor n_1/c \rfloor > 1 \\
\mbox{$[X_1^{2},I_{n_2-2c}]$} & \lfloor n_1/c \rfloor=1, \, n_1 > c  \\
\mbox{$[X_1,I_{n_2-c}]$} & n_1= c
\end{array}\right.$$
then $\la x \ra$ is not $T_2$-conjugate to any of the subgroups in \eqref{e:subg}.

Now assume $c_1 >c_2$, in which case one of the following holds:
\begin{itemize}\addtolength{\itemsep}{0.2\baselineskip}
\item[{\rm (a)}] $T_1 \ne {\rm PSL}_{n_1}(q)$, $T_2 = {\rm PSL}_{n_2}(q)$, $i \geqs 3$ is odd, $c_1 = 2i$, $c_2 = i$;
\item[{\rm (b)}] $T_1 \ne {\rm PSU}_{n_1}(q)$, $T_2 = {\rm PSU}_{n_2}(q)$, $i \equiv 2 \imod{4}$, $i \geqs 6$, $c_1 = i$, $c_2 = i/2$.
\end{itemize}
Set $s = (r-1)/c_1$ and $t = (r-1)/c_2$, so $t = 2s$. As before, let $\{\la x_j \ra \,:\, 1 \leqs j \leqs \kappa(T_1,r)\}$ be a set of representatives of the $T_1$-classes of subgroups of order $r$, where $x_j = \hat{x}_jZ$ and $\hat{x}_j$ is given in \eqref{e:xj}.
Now every element $y \in T_2$ of order $r$ is of the form $y = \hat{y}Z$ with 
$$\hat{y} = [Y_1^{a_1}, \ldots, Y_{s}^{a_{s}},Y_{s+1}^{a_{s+1}}, \ldots, Y_{t}^{a_{t}},I_{e'}].$$
Without loss of generality, we may assume that for each $j \in \{1, \ldots, s\}$, the set of eigenvalues of $X_j$ (in $\mathbb{F}_{q^i}$) is the union of the eigenvalues of $Y_j$ and $Y_{s+j}$. We can now repeat the argument for the case $c_1=c_2$, replacing each $X_m$ by $Y_m$. The result follows.
\end{proof}

\begin{lem}\label{l:prel1}
Let $T = {\rm PSL}_{n}^{\e}(q)$ and let $r \geqs 5$ be a prime divisor of $q^2-1$. Define $c$ as in \eqref{e:eqcc}. 
\begin{itemize}\addtolength{\itemsep}{0.2\baselineskip}
\item[{\rm (i)}] If $(n,c)=(3,1)$ then $\kappa(T,r) \leqs r-1$.
\item[{\rm (ii)}] If $(n,c)=(4,1)$ then $\kappa(T,r) \leqs (r^2-3r+6)/2$.
\item[{\rm (iii)}] If $(n,c)=(6,2)$ then $\kappa(T,r) \leqs (r^2+15)/8$.
\end{itemize}
\end{lem}

\begin{proof}
Write ${\rm PGL}_{n}^{\e}(q) = {\rm GL}_{n}^{\e}(q)/Z$ and let $\omega \in \mathbb{F}_{q^2}$ and $x \in T$ be elements of order $r$. Since $r \geqs 5$ and $n \in \{3,4,6\}$ we have $(r,n)=1$ so we may write $x = \hat{x}Z$ with $\hat{x} \in {\rm GL}_{n}^{\e}(q)$ of order $r$ (see \cite[Lemma 3.11]{Bur2}).

First assume $(n,c)=(3,1)$. By replacing $x$ by a suitable conjugate, we may assume $\hat{x} = [1,\l_1,\l_2] \in {\rm GL}_{3}^{\e}(q)$, where $\l_1 \ne \l_2$ and $\l_2 \ne 1$. Clearly, if $\l_1=1$ then $\la x \ra$ is $T$-conjugate to $\la [1,1,\omega]Z \ra$. On the other hand, if $\l_1 \neq 1$ then $\la x \ra$ is $T$-conjugate to $\la [1,\omega,\omega^{j}]Z \ra$ for some $1<j<r$. The result follows. Similarly, if $(n,c)=(4,1)$ then any subgroup of $T$ of order $r$ is conjugate to one of the following:
$$\la [1,1,1,\omega]Z \ra,\; \la [1,1,\omega,\omega^{j}]Z \ra,\; \la [1,\omega,\omega^{k},\omega^{k'}]Z \ra,$$
where $1 \leqs j<r$ and $1<k<k'<r$. Therefore, there are at most 
$$1+(r-1)+\binom{r-2}{2} = (r^2-3r+6)/2$$
such classes. Finally, suppose $(n,c)=(6,2)$. Set $s=(r-1)/2$ and write 
$$\hat{x} = [X_1^{a_1}, \ldots, X_s^{a_s},I_e]$$ 
as in Remark \ref{r:cc}. Then the $T$-classes of subgroups of order $r$ are represented by 
$$\la [X_1,I_4]Z \ra, \; \la [X_1,X_j,I_2]Z \ra, \; \la [X_1^2,X_j]Z \ra,\; \la [X_1,X_{k},X_{k'}]Z\ra,$$
where $1 \leqs j \leqs (r-1)/2$ and $1 <k <k' \leqs (r-1)/2$. Therefore, there are at most
$$r+\binom{(r-3)/2}{2} = (r^2+15)/8$$
such classes, as claimed.
\end{proof}

\subsection{Subgroup structure}

Let $G$ be an almost simple classical group over $\mathbb{F}_{q}$ with socle $T$ and natural module $V$. Set $n = \dim V$ and let $H$ be a maximal subgroup of $G$ with $G = HT$. Recall that Aschbacher's subgroup structure theorem states that either $H$ belongs to one of eight geometric subgroup collections, or $H$ is almost simple and acts irreducibly on $V$. The latter collection of non-geometric subgroups is denoted by $\mathcal{S}$, and the formal definition of this collection is as follows (see \cite[p.3]{KL}). Note that the various conditions are designed to ensure that a subgroup in $\mathcal{S}$ is not contained in one of the geometric subgroup collections.

\begin{defn}\label{sdef}
A subgroup $H$ of $G$ belongs to the collection $\mathcal{S}$ if and only if it satisfies the following conditions:
\begin{itemize}\addtolength{\itemsep}{0.2\baselineskip}
\item[(i)] The socle $S$ of $H$ is a nonabelian simple group and $S \not\cong T$.
\item[(ii)] If $\hat{S}$ is the full covering group of $S$, and if $\rho : \hat{S} \to {\rm GL}(V)$ is a representation of $\hat{S}$ such that, modulo scalars, $\rho(\hat{S})=S$, then $\rho$ is absolutely irreducible.
\item[(iii)] $\rho(\hat{S})$ cannot be realised over a proper subfield of $\mathbb{F}$, where 
$\mathbb{F} = \mathbb{F}_{q^2}$ if $T={\rm PSU}_{n}(q)$, otherwise $\mathbb{F} = \mathbb{F}_{q}$.
\item[(iv)] If $\rho(\hat{S})$ fixes a nondegenerate quadratic form on $V$ then $T={\rm P\Omega}_{n}^{\e}(q)$.
\item[(v)] If $\rho(\hat{S})$ fixes a nondegenerate alternating form on $V$, but no nondegenerate quadratic form, then $T={\rm PSp}_{n}(q)$.
\item[(vi)] If $\rho(\hat{S})$ fixes a nondegenerate hermitian form on $V$ then $T={\rm PSU}_{n}(q)$.
\item[(vii)] If $\rho(\hat{S})$ does not fix a form as in (iv), (v) or (vi) then $T={\rm PSL}_{n}(q)$.
\end{itemize}
\end{defn}

Let $x \in G \cap {\rm PGL}(V)$ be a nontrivial element and write $x = \hat{x}Z$, where $V$ is the natural module for $T$, $\hat{x} \in {\rm GL}(V)$ and $Z = Z({\rm GL}(V))$. Set $\bar{V} = V \otimes \bar{\mathbb{F}}_q$, where $\bar{\mathbb{F}}_q$ is the algebraic closure of $\mathbb{F}_q$, and define 
\begin{equation}\label{e:nu}
\nu(x) = \min\{\dim [\bar{V},\l\hat{x}] \,:\, \l \in \bar{\mathbb{F}}_q^{\times}\}
\end{equation}
where $[\bar{V},\l\hat{x}]$ is the subspace $\la v-v^{\l\hat{x}} \mid v \in \bar{V} \ra$. 
Note that $\nu(x)$ is the codimension of the largest eigenspace of $\hat{x}$ on $\bar{V}$.  

The following theorem is a special case of \cite[Theorem 7.1]{GS} 
(recall that the subgroups in the collections $\mathcal{A}$ and $\mathcal{B}$ are recorded in Tables \ref{atab} and \ref{btab}, respectively).

\begin{thm}\label{gursax}
Let $G$ be a finite almost simple classical group with socle $T$ and let $H \in \mathcal{S}$ be a subgroup of $G$. Let $n$ be the dimension of the natural module for $T$, and assume that $n \geqs 6$ and $H \not\in \mathcal{A} \cup \mathcal{B}$. Then 
$$\nu(x) > \max\{2,\sqrt{n}/2\}$$
for all nontrivial $x \in H \cap {\rm PGL}(V)$.
\end{thm}

This result plays a central role in our proof of Theorem \ref{t:main}. First we handle the excluded cases; the relevant $r$-elusive groups with $n<6$ were determined in \cite{BG} (see Table \ref{t:lowdim}), and the groups with a point stabiliser in $\mathcal{A}$ or $\mathcal{B}$ will be handled in the next two sections. At this point we are in a position to apply Theorem \ref{gursax}, which immediately implies that any element $x \in T$ of order $r$ with $\nu(x) \leqs \max\{2,\sqrt{n}/2\}$ is a derangement. In this way, we quickly reduce to the case $r \ne p$, $r \geqs 5$ and $c > \max\{2,\sqrt{n}/2\}$, where $c$ is the integer in \eqref{e:eqcc}. Moreover, we may assume that $r$ divides $|H \cap T|$. If $c > n/2$ then $T$ is $r$-elusive by Corollary \ref{p:star}, so we can assume that 
$$\max\{2,\sqrt{n}/2\} < c \leqs n/2$$
and our goal will be to show that $T$ contains a derangement of order $r$. This final step will be   carried out in Section \ref{s:c}.

\section{The collection $\mathcal{A}$}\label{s:a}

Let $G \leqs {\rm Sym}(\Omega)$ be an almost simple primitive classical group over $\mathbb{F}_{q}$ with socle $T$ and point stabiliser $H \in \mathcal{S}$. Let $S$ denote the socle of $H$ and let $V$ be the natural $T$-module. Recall that $V$ is absolutely irreducible as an $\hat{S}$-module, where $\hat{S}$ is an appropriate covering group of $S$. In this section we investigate the special case where $H$ belongs to the collection $\mathcal{A}$. Here $S=A_d$ is the alternating group of degree $d$ and $V$ is the \emph{fully deleted permutation module} for $S$ over $\mathbb{F}_{p}$. The relevant cases that arise are recorded in Table \ref{atab}.

We begin by recalling the construction of $V$. Let $p$ be a prime, let $d \geqs 5$ be an integer and consider the permutation module $\mathbb{F}_{p}^d$ for $S_d$. Define subspaces  
$$U=\{(a_1, \ldots, a_d) \,:\, \sum_{i=1}^{d}a_i=0\},\;\; W=\{(a,\ldots, a) \,:\, a \in \mathbb{F}_{p}\}$$
of $\mathbb{F}_{p}^d$, and observe that $U$ and $W$ are the only nonzero proper $A_d$-invariant submodules of 
$\mathbb{F}_{p}^d$. Then
$V=U/(U \cap W)$ is the fully deleted permutation module for $A_d$, which is an absolutely irreducible $A_d$-module over $\mathbb{F}_{p}$. Set $n = \dim V$ and note that 
$n = d-2$ if $p$ divides $d$, otherwise $n = d-1$. Note that $A_d$ preserves the symmetric bilinear form $B':U \times U \to \mathbb{F}_{p}$ defined by
$$B'((a_1, \ldots, a_d),(b_1, \ldots, b_d)) = \sum_{i=1}^{d}a_ib_i$$
and thus $B'$ induces a symmetric bilinear form $B$ on $V$. By \cite[Proposition 5.3.5]{KL}, if $d \geqs 10$ then $V$ has the smallest dimension of all nontrivial irreducible $A_d$-modules over $\mathbb{F}_{p}$.

Suppose $p$ is odd. In this situation, the $A_d$-module $V$ affords an embedding of $A_d$ into an orthogonal group $\Omega_{n}^{\e}(p)$. By choosing a suitable basis for $V$ it is straightforward to compute the determinant of the Gram matrix of $B$, and subsequently the discriminant $D(Q) \in \{\square, \boxtimes\}$ of the 
corresponding quadratic form $Q$ on $V$ (which is defined by $Q(v)=\frac{1}{2}B(v,v)$ for $v \in V$). 

For example, suppose $d$ is even and $p$ divides $d$, so $n=d-2$ and $U \cap W = W$. Let $\{v_1, \ldots, v_d\}$ be the standard basis for $\mathbb{F}_{p}^d$ and set $e_{i}=(v_{i}-v_{i+1})+W$, $1 \leqs i \leqs n$. Then 
\begin{equation}\label{e:beta}
\b=\{e_1, \ldots, e_{n}\}
\end{equation}
is a basis for $V$ and 
$$J_{\b}=\left(\begin{array}{ccccc}
2 & -1 & & &  \\
-1 & 2 & -1 & & \\
 & & \ddots & &  \\
& &  -1 & 2 & -1 \\
& & & -1 & 2 
\end{array}\right)$$
is the corresponding Gram matrix of $B$. Therefore $\det(J_{\b})=n+1$, so $D(Q)=\square$ if $n+1$ is a square in $\mathbb{F}_{p}$, otherwise $D(Q)=\boxtimes$. 

In general, if $p$ is odd and $n$ is even then using \cite[Proposition 2.5.10]{KL} we calculate that $\e=+$ if and only if
$$\left(\frac{n+1}{p}\right) = (-1)^{\frac{1}{4}n(p-1)}$$
where the term on the left is the \emph{Legendre symbol} (which takes the value $1$ if $n+1$ is a quadratic residue modulo $p$, $0$ if $p$ divides $n+1$, and $-1$ in the remaining cases; here $n+1$ is indivisible by $p$, so it is always nonzero). Note that if $d$ is even and $p$ divides $d$ then 
\begin{equation}\label{e:note}
\left(\frac{n+1}{p}\right) =\left(\frac{-1}{p}\right) = (-1)^{\frac{1}{2}(p-1)}
\end{equation}
and thus $\e=-$ if and only if $d \equiv 2 \imod{4}$ and $p \equiv 3 \imod{4}$.

Now assume $p=2$ so $n$ is even. Let $u=(a_1, \ldots, a_d) \in U$. We define a map $Q':U \to \mathbb{F}_{2}$ by setting $Q'(u)=1$ if the number of nonzero $a_i$ is congruent to $2$ modulo $4$, otherwise $Q'(u)=0$. Then $Q'$ is an $A_d$-invariant quadratic form on $U$ with associated bilinear form $B'$. If $d \not\equiv 2 \imod{4}$ then $Q'$ induces a nondegenerate quadratic form $Q$ on $V$, so in this case we obtain an embedding $A_d \leqs \Omega_{n}^{\e}(2)$ where $\e$ is given in Table \ref{atab} (see \cite[p.187]{KL}). On the other hand, if $d \equiv 2 \imod{4}$ then $A_d$ does not fix a nondegenerate quadratic form on $V$, so we have an embedding $A_{d} \leqs {\rm Sp}_{d-2}(2)$.

The specific irreducible embeddings that arise in this way are listed in Table \ref{atab}. Note that the conditions on $d$ in the final column ensure that $S=A_d$ is simple and not isomorphic to $T$. For the remainder of this section we set $H_0 = H \cap T$.

\begin{lem}\label{l:acol}
We have $H_0 = S_d$ if and only if $T = {\rm Sp}_{n}(2)$, or $np$ is odd and $\left(\frac{(n+1)/2}{p}\right) = 1$.
\end{lem}

\begin{proof}
Let $x$ be the transposition $(1,2)$ in $S_d$. If $p=2$ then $x$ has Jordan form $[J_2,J_1^{n-2}]$ on $V$, so $x \in T$ if and only if $T$ is a symplectic group. Now assume $p$ is odd, so $T$ is an orthogonal group. Up to conjugacy, $x$ acts on $V$ as a diagonal matrix $[-I_{1},I_{n-1}]$ (modulo scalars), so $x \in T$ only if $n$ is odd. In terms of the above basis $\b$ for $V$ (see \eqref{e:beta}), $x$ maps $e_1$ to $-e_1$, $e_2$ to $e_1+e_2$, and it fixes all the other basis vectors. Then
$E = \la e_1+2e_2, e_3, \ldots, e_n\ra$ is the $1$-eigenspace of $x$, which is a nondegenerate $(n-1)$-space of type $\e'$. To determine whether or not $x \in T$ we need to calculate $\e'$. 

It is straightforward to check that the Gram matrix of the induced bilinear form on $E$ has determinant $(n+1)/2$, so \cite[Proposition 2.5.10]{KL} implies that $\e'=+$ if and only if
$$\left(\frac{(n+1)/2}{p}\right) = (-1)^{\frac{1}{4}(n-1)(p-1)}$$
If $\e'=+$ (respectively, $\e'=-$) then $x \in {\rm SO}_{n}(p)$ is an involution of type $t_{(n-1)/2}$ (respectively, $t_{(n-1)/2}'$) in the notation of \cite{BG, GLS}, and the desired result follows by inspecting \cite[Table 4.5.1]{GLS}. For example, if $\e'=+$ then we find that an involution in ${\rm SO}_{n}(p)$ of type $t_{(n-1)/2}$ is in $T$ if and only if 
$$p^{\frac{1}{2}(n-1)} \equiv 1 \imod{4},$$ 
whence $H_0 = S_d$ if and only if $\left(\frac{(n+1)/2}{p}\right) = 1$.
\end{proof}

In the statement of the next lemma, we use the notation in \eqref{e:unip} for expressing the Jordan form of an element of order $p$.

\begin{lem}\label{lem:jordan}
Let $x \in S_d$ be an element of order $p$ with cycle-shape $(p^h,1^s)$. Then the Jordan form of $x$ on $V$ is as follows:
\begin{itemize}\addtolength{\itemsep}{0.2\baselineskip}
\item[{\rm (i)}] $[J_p^h,J_1^{s-1}]$ if $s\geqs 1$ and $(p,d)=1$;
\item[{\rm (ii)}] $[J_p^h,J_1^{s-2}]$ if $s\geqs 1$ and $p$ divides $d$;
\item[{\rm (iii)}] $[J_p^{h-1},J_{p-2}]$ if $s=0$ and  $(p,h)=1$;
\item[{\rm (iv)}] $[J_p^{h-2},J_{p-1}^2]$ if $s=0$, $p$ divides $h$, and $h\neq 2$;
\item[{\rm (v)}] $[J_2]$ if $s=0$ and $p=h=2$.
\end{itemize}
\end{lem}

\begin{proof}
Up to conjugacy, we may assume that 
$$x=(1,\ldots,p)\cdots ((h-1)p+1,\ldots, hp).$$
Suppose first that $s \geqs 1$. Then for each $i \in \{0,\ldots,h-1\}$,  
$$\mathcal{E}_i=\{e_{ip+1}-e_d+(U\cap W), \ldots, e_{(i+1)p}-e_d+(U\cap W)\}$$ 
is a set of $p$ linearly independent vectors in $V$, which are cyclically permuted by $x$, and $\mathcal{E}_0\cup\ldots\cup\mathcal{E}_{h-1}$ is a linearly independent set of $hp$ vectors. Therefore, \cite[Lemma 5.2.6]{BG} implies that $x$ has Jordan form $[J_p^h,J_1^{s-1}]$ if $(p,d)=1$ and $[J_p^h,J_1^{s-2}]$ if $p$ divides $d$.

For the remainder, let us assume that $s=0$, so $n=d-2$, $U \cap W = W$ and $x$ cyclically permutes the $p$ vectors 
$$\{e_1-e_2+W,\ldots, e_{p-1}-e_p+W, e_p-e_1+W\}.$$ 
If $h=1$ then $V$ is spanned by this set of vectors and the first $p-2$ form a basis for $V$. Thus $x$ has Jordan form $[J_{p-2}]$ on $V$.  Suppose now that $h \geqs 2$.  Then for each $i\in\{1, \ldots, h-1\}$ the set 
$$\mathcal{E}_i=\{ e_1-e_{ip+1}+W,e_2-e_{ip+2}+W, \ldots, e_p-e_{(i+1)p}+W\}$$
is a set of $p$ linearly independent vectors cyclically permuted by $x$. If $(p,h)=1$ then $\mathcal{E}_1\cup\ldots\cup\mathcal{E}_{h-1}$ is an $x$-invariant set of linearly independent vectors and by Lemma \cite[Lemma 5.2.6]{BG}, the Jordan form of $x$ on the span of these vectors is $[J_p^{h-1}]$. By \cite[Lemma 4.3]{PSY}, the $1$-eigenspace of $x$ on $V$ has dimension $h$, so it follows that $x$ has Jordan form $[J_p^{h-1}, J_{p-2}]$ on $V$.

Now assume $p$ divides $h$. If $p=h=2$ then $\dim V=2$ and $x$ acts nontrivially on $V$, so $x$ has  Jordan form $[J_2]$. Now assume $h \geqs 3$. Note that 
$\mathcal{E}_1 \cup\ldots\cup\mathcal{E}_{h-1}$ is linearly dependent, whereas $\mathcal{E}=\mathcal{E}_1 \cup \ldots \cup \mathcal{E}_{h-2}$ is linearly independent.  Let $Y$ be the span of $\mathcal{E}$. Now $x$ cyclically permutes the $p$ vectors 
$$\{e_{(h-1)p+1} -e_{(h-1)p+2}+W, \ldots, e_{(h-1)p+p+1}-e_{hp}+W,e_{hp}-e_{(h-1)p+1}+W\}$$
which span a $(p-1)$-dimensional subspace $Z$ of $V$ such that $Y \cap Z = 0$. Moreover, the Jordan form of $x$ on $Z$ is $[J_{p-1}]$.  By \cite[Lemma 4.3]{PSY}, the $1$-eigenspace of $x$ on $V$ is $h$-dimensional, and so $x$ has Jordan form $[J_{p-1}^{h-2},J_{p-1}^2]$ on $V$.
\end{proof}

\begin{lem}\label{lem:semi}
Let $x \in S_d$ be an element of prime order $r \ne p$ with cycle-shape $(r^h,1^s)$ and consider the action of $x$ on $\bar{V} = V \otimes \mathbb{F}$, where $\mathbb{F} = \bar{\mathbb{F}}_{p}$. Then every nontrivial $r$-th root of unity occurs as an eigenvalue of $x$ on $\bar{V}$ with multiplicity $h$.
\end{lem}

\begin{proof}
Let $\mathbb{F}^d$ be the permutation module for $S_d$ over $\mathbb{F}$ and set $\bar{U} = U \otimes \mathbb{F}$ and $\bar{W} = W \otimes \mathbb{F}$. Let $\omega \in \mathbb{F}$ be a nontrivial $r$-th root of unity. By \cite[Lemma 5.2.6]{BG}, $\omega$ occurs as an eigenvalue of $x$ on $\mathbb{F}^d$ with multiplicity $h$. If $p$ does not divide $d$ then $\mathbb{F}^d=\bar{U} \oplus \bar{W}$, $\bar{V}=\bar{U}$ and $\bar{W}$ is contained in the $1$-eigenspace of $x$ on $\mathbb{F}^d$. Therefore, $\omega$ has multiplicity $h$ as an eigenvalue of $x$ on $\bar{V}$.  

Now assume $p$ divides $d$, so $\bar{W} \leqs \bar{U}$ and $\bar{V}=\bar{U}/\bar{W}$. Now $x$ has a fixed point and without loss of generality we may assume that $x$ fixes the standard basis element $v_d \in \mathbb{F}^d$. Since $\mathbb{F}^d=\bar{U} \oplus \langle v_d \rangle$ and $v_d$ is a $1$-eigenvector for $x$, it follows that $\omega$ has multiplicity $h$ as an eigenvalue of $x$ on $\bar{U}$. Since $\bar{W}$ is also contained in the $1$-eigenspace of $x$ we conclude that $\omega$ has multiplicity $h$ as an eigenvalue of $x$ on $\bar{V}$.
\end{proof}

We are now ready to state the main result of this section. In the proof of the proposition, we freely use the notation for prime order elements introduced in Section \ref{ss:cc}, which is consistent with the notation adopted in \cite{BG}. In part (ii) of the statement, we define the integer $c$ as in \eqref{e:eqcc}.

\begin{prop}\label{p:a}
Let $G \leqs {\rm Sym}(\Omega)$ be a primitive almost simple classical group over $\mathbb{F}_{q}$ with socle $T$ and point stabiliser $H \in \mathcal{A}$. Let $r$ be a prime divisor of $|\Omega|$ and assume that $n \geqs 6$. Set $H_0 = H \cap T$ and note that $q=p$ is a prime. Then $T$ is $r$-elusive if and only if one of the following holds:
\begin{itemize}\addtolength{\itemsep}{0.2\baselineskip}
\item[{\rm (i)}] $r=2$, $p \neq 2$ and either
\begin{itemize}\addtolength{\itemsep}{0.2\baselineskip}
\item[{\rm (a)}] $T = \Omega_n(p)$ and $\left(\frac{(n+1)/2}{p}\right) = 1$; or
\item[{\rm (b)}] $T = {\rm P\Omega}_{n}^{\e}(p)$, $n \equiv 2\imod{4}$ and $p \equiv 5\e \imod{8}$. 
\end{itemize}
\item[{\rm (ii)}] $r \ne p$, $r>2$, $r$ divides $|H_0|$ and $c=r-1$.
\end{itemize}  
\end{prop}

\begin{proof}
Here $H_0 \in \{A_{d},S_{d}\}$ and $d \geqs 5$. If $r=p>2$ then $[J_{2}^{2},J_1^{n-4}] \in T$ is a derangement by Lemma \ref{lem:jordan}. Now assume $r=p=2$ and note that by Lemma \ref{lem:jordan}, $x=(1,2)(3,4) \in H_0$ has Jordan form $[J_2^2,J_1^{n-4}]$ on $V$. Moreover, in terms of the Aschbacher-Seitz \cite{AS} notation, we identify $x$ as a $c_2$-type involution since
$$B(e_3+e_4,(e_3+e_4)x) = B(v_3+v_5,v_4+v_5) = 1.$$
We conclude that the $a_2$-type involutions in $T$ are derangements.

Next suppose $r \neq p$ and $r>2$. Let $i = \Phi(r,p)$ (see \eqref{e:phirq}), so $i$ is the smallest positive integer such that $r$ divides $p^i-1$. Clearly, if $r$ fails to divide $|H_0|$ then every element in $T$ of order $r$ is a derangement, so let us assume $r$ divides $|H_0|$. Let $x \in H_0$ be an element of order $r$ and write $x=\hat{x}Z$, where $\hat{x} \in {\rm GL}_{n}(p)$ has order $r$. By Lemma \ref{lem:semi}, the multiset of eigenvalues of $\hat{x}$ on $\bar{V} = V \otimes \bar{\mathbb{F}}_{q}$ contains every nontrivial $r$-th root of unity with equal multiplicity. Therefore, if $i$ is even and $i<r-1$ then $[\L,I_{n-i}]$ is a derangement. Similarly, if $i$ is odd and $i<(r-1)/2$ then $[\L,\L^{-1},I_{n-2i}]$ has the desired property. Now assume $i$ is even and $i=r-1$, so $\hat{x}$ is conjugate to an element of the form $[\L^{h},I_{n-h(r-1)}]$ for some $h \geqs 1$ with $hr \leqs n$. There is a unique $T$-class of such elements for each value of $h$, and $x^T \cap H$ consists of the permutations in $H_0$ with cycle-shape $(r^{h},1^{d-hr})$. In particular, $T$ is $r$-elusive. An entirely similar argument applies if $i=(r-1)/2$ is odd.

To complete the proof of the proposition, we may assume that $r=2$ and $p \neq 2$, so $T$ is an orthogonal group (see Table \ref{atab}). By Lemma \ref{lem:semi}, if $x \in S_d$ has cycle-shape $(2^h,1^s)$ then the $(-1)$-eigenspace of $x$ on $V$ has dimension $h\leqs d/2$. 

Suppose first that $T={\rm P\Omega}^+_n(q)$. If $n\equiv 0\pmod 4$ then $T$ contains involutions of type $t_{n/2}$ or $t_{n/2}'$, and these elements are derangements because they do not have $-1$ as an eigenvalue (see \cite[Sections 3.5.2.10 and 3.5.2.11]{BG}). Now assume $n\equiv 2\pmod 4$. If $p \equiv 1\pmod 8$ then the same argument implies that involutions of type $t_{n/2}$ in $T$ are derangements. If $p \equiv 3\pmod 4$ then $T$ contains two classes of involutions (namely, $t_1$ and $t_1'$) with a $2$-dimensional $(-1)$-eigenspace and so one of these classes must consist of derangements. This leaves $p \equiv 5\pmod{8}$, in which case $H_0=A_d$ by Lemma \ref{l:acol}. Here every involution in $T$ has a $2\ell$-dimensional $(-1)$-eigenspace for some $1 \leqs \ell < n/4$ (see \cite[Table B.10]{BG}), and there is a unique class of such involutions for each $\ell$. We conclude that $T$ is $2$-elusive. 
A very similar argument applies if $T = {\rm P\Omega}^-_n(q)$ and we omit the details.

Finally, suppose $T=\Omega_n(p)$ with $n$ odd. Here every involution in $T$ is of the form $[-I_{2\ell},I_{n-2\ell}]$, and there is a unique such class for each $1 \leqs \ell \leqs (n-1)/2$ (see \cite[Table B.8]{BG}). Now, if $H_0=A_d$ then $H_0$ does not contain a transposition, so any involution in $T$ of the form $[-I_{n-1},I_{1}]$ is a derangement. On the other hand, if $H_0=S_d$ then it is easy to see that every involution in $T$ has fixed points, so $T$ is $2$-elusive.  Note that $H_0 = S_d$ if and only if $\left(\frac{(n+1)/2}{p}\right) = 1$ (see Lemma \ref{l:acol}).
\end{proof}

\section{The collection $\mathcal{B}$}\label{s:b}

In this section we turn our attention to the case where $H \in \mathcal{S}$ is a subgroup in  
the collection $\mathcal{B}$ (see Table \ref{btab}). Recall that these cases arise naturally as exceptions in the statement of Theorem \ref{gursax}, so $n \geqs 6$ and  
$$\nu(x) \leqs \max\{2,\sqrt{n}/2\}$$
for some nontrivial element $x \in H \cap {\rm PGL}(V)$. Our main result is the following (note that Table \ref{t:bex} is located in the introduction).

\begin{prop}\label{p:b}
Let $G \leqs {\rm Sym}(\Omega)$ be a primitive almost simple classical group with socle $T$ and point stabiliser $H \in \mathcal{B}$. Let $r$ be a prime divisor of $|\Omega|$ and let $S$ denote the socle of $H$. Then $T$ is $r$-elusive if and only if $(T,S,r)$ is one of the cases listed in Table \ref{t:bex}.
\end{prop}

\begin{rem}\label{r:42}
The conditions recorded in the final column of Table \ref{t:bex} are needed to ensure that every element in $T$ of order $r$ has fixed points, and they also imply that $r$ divides the degree of $G$. Note that these conditions are additional to the ones given in Table \ref{btab}, which are needed for the existence and maximality of $H$ in $G$. We refer the reader to the tables in \cite[Section 8.2]{BHR} for the precise conditions required for maximality, and for a detailed description of the structure of $H_0 = H \cap T$. Further information on these cases can be found in \cite[Section 2.3]{Bur6}. Also note that the relevant covering group $\hat{S}$ is given in the statement of \cite[Theorem 7.1]{GS}.
\end{rem}

\begin{lem}\label{l:b1}
Proposition \ref{p:b} holds in Case $(\mathcal{B}1)$ of Table \ref{btab}.
\end{lem}

\begin{proof}
Here $T={\rm PSp}_{10}(p)$, $S={\rm PSU}_{5}(2)$ and $p \ne 2$. According to \cite[Table 8.65]{BHR}, we have $H_0 = S.2$ if and only if $p \equiv \pm 1\imod{8}$. Let $r$ be a prime divisor of $|\Omega|$. If $r$ does not divide $|H_0|$ then any element in $T$ of order $r$ is a derangement, so we may as well assume that $r$ also divides $|H_0|$, hence $r \in \{2,3,5,11\}$. 

If $r=p$ then $H_0$ has at most six classes of elements of order $r$, but $T$ has at least seven by \cite[Proposition 3.4.10]{BG} and thus $T$ is not $r$-elusive by Corollary \ref{c:count}. Now assume $r \neq p$. Set $i=\Phi(r,p)$ as in \eqref{e:phirq} and define $\nu(x)$ for $x \in T$ as in \eqref{e:nu}. Let $\chi$ be the corresponding Brauer character of $H_0$ (this is available in \textsf{GAP} \cite{GAP4}, for example). One observes that $\{\chi(x) \,:\, x \in H_0,\, |x|=3\} = \{-5,-2,1,4\}$, which implies that every $x \in T$ of order $3$ with $\nu(x)=2$ is a derangement (indeed, over $\bar{\mathbb{F}}_{p}$ such an element has eigenvalues $\omega$, $\omega^2$ and $1$ (the latter with multiplicity $8$), so $\chi(x)=7$). In the same way, we deduce that the elements $x \in T$ of order $5$ with $\nu(x)=4$ are derangements. If $r=11$ then $i \in \{1,2,5,10\}$ and by considering $\chi$ we see that $T$ is $11$-elusive if and only if $i>2$ (in fact, we need the condition $p^5 \equiv \pm 1 \imod{121}$ to ensure that $|\Omega|$ is divisible by $11$).

Finally, let us assume $r=2$. By inspecting the values of $\chi$ we deduce that the involutions $x \in T$ with $\nu(x)<5$ have fixed points, whereas those with $\nu(x)=5$ have fixed points if and only if $H_0=S.2$ (in this situation, $H_0$ contains an involutory graph automorphism $\gamma$ of $S$ such that $\nu(\gamma)=5$). We conclude that $T$ is $2$-elusive if and only if $p \equiv \pm 1\imod{8}$.
\end{proof}

\begin{lem}\label{l:b2}
Proposition \ref{p:b} holds in Case $(\mathcal{B}2)$ of Table \ref{btab}.
\end{lem}

\begin{proof}
Here $T = {\rm P\Omega}_{8}^{+}(q)$ and $H_0=\Omega_7(q)$ if $q$ is odd, otherwise $H_0 = {\rm Sp}_{6}(q)$. This embedding arises by restricting an irreducible spin representation of $\Omega_{8}^{+}(q)$ to the stabiliser of a $1$-dimensional nonsingular subspace of the natural $\Omega_{8}^{+}(q)$-module. Let $r$ be a prime divisor of $|H_0|$ and $|\Omega|$.

First assume $q$ is even, so $H_{0}={\rm Sp}_{6}(q)$. By inspecting the proof of \cite[Lemma 2.7]{Bur6}, we deduce that every $c_2$-type involution in $T$ is a derangement (here we are using the standard Aschbacher-Seitz \cite{AS} notation for involutions). 
Now assume $r$ is odd. Let $i=\Phi(r,q)$, so $i \in \{1,2,4\}$. 
If $i \in \{1,2\}$ then the proof of \cite[Lemma 2.7]{Bur6} indicates that every element $x \in T$ of order $r$ with $\nu(x)=2$ is a derangement. Similarly, if $i=4$ then the elements with $\nu(x)=4$ are derangements.

A very similar argument applies when $q$ is odd. For example, the proof of \cite[Lemma 2.7]{Bur6} shows that $[J_{3},J_1^5]$ and $[-I_{2},I_{6}]$ are derangements in $T$ of order $p$ and $2$, respectively. Finally, if $r \neq p$ and $r>2$ then we can proceed as above in the $q$ even case.
\end{proof}

\begin{lem}\label{l:b3}
Proposition \ref{p:b} holds in Case $(\mathcal{B}3)$ of Table \ref{btab}.
\end{lem}

\begin{proof}
Here $T={\rm P\Omega}_{8}^{+}(q)$ and $H_{0}=C_{T}(\psi) = {}^{3}D_4(q_0)$, where $q=q_0^3$ and $\psi$ is a triality graph-field automorphism of $T$. In view of the proof of \cite[Lemma 2.12]{Bur6}, this characterisation of $H_0$ implies that if $p \neq 2$ then unipotent elements with Jordan form 
$[J_{3},J_1^{5}]$ are derangements of order $p$, and so are involutions of type $a_{4}$ when $p=2$. Similarly, if $p \neq 2$ then the involutions of type $[-I_{2},I_{6}]$ are also derangements.

Let $r \ne p$ be an odd prime divisor of $|\Omega|$ and $|H_0|$. Set $i=\Phi(r,q)$ and note that $i \in \{1,2,4\}$. Let $x \in T$ be an element of order $r$ with $\nu(x)=\a$, where $\a=2$ if $i \in \{1,2\}$, otherwise $\a=4$. Then $x$ is not centralised by $\psi$ (see \cite[Proposition 3.55(iv)]{Bur2}), so $x$ is a derangement. For example, if $i \in \{1,2\}$ and $\nu(x)=2$ then $\nu(x^{\psi}) = 4$.
\end{proof}

\begin{lem}\label{l:b4}
Proposition \ref{p:b} holds in Case $(\mathcal{B}4)$ of Table \ref{btab}.
\end{lem}

\begin{proof}
Here $T={\rm P\Omega}_{8}^{+}(p)$, $H_{0}=\Omega_{8}^{+}(2)$ and $p \ne 2$ (see \cite[Table 8.50]{BHR}). Let $r$ be a prime divisor of $|\Omega|$ and $|H_0|$, so $r \in \{2,3,5,7\}$. If $r=p$ then $H_0$ has at most five classes of subgroups of order $r$, whereas $T$ has at least six (see \cite[Proposition 3.5.12]{BG}). Now assume $r \neq p$ and note that $p \geqs 7$ (indeed, if $p \in \{3,5\}$ then $p$ is the only prime dividing $|\Omega|$ and $|H_0|$). Set $i = \Phi(r,p)$.

Suppose $p=7$, so $r \in \{2,5\}$. If $r=5$ then $i=4$ and we deduce that $T$ is $5$-elusive by considering the values of the corresponding Brauer character $\chi$ of $2.\Omega_{8}^{+}(2)$. Now assume $r=2$. The involutions in $T$ are of type $t_1',t_2,t_3'$ and $t_4'$, in terms of the notation in \cite{GLS, BG}. By inspecting $\chi$ we see that the $t_1'$ elements have fixed points, and so do the involutions in at least one of the other classes. Since $H_0$ is normalised by a triality graph automorphism $\tau$ of $T$, and $\tau$ permutes the $T$-classes represented by the elements $t_1',t_3',t_4'$, we conclude that every involution in $T$ has fixed points, so $T$ is $2$-elusive.

Now assume $p>7$. As above, $T$ is $2$-elusive. By considering $\chi$ we see that every element of order $3$ has fixed points, and we note that $|\Omega|$ is divisible by $3$ if and only if $p^2 \equiv 1 \imod{9}$. Similarly, if $r \in \{5,7\}$ then $T$ is $r$-elusive if and only if $i>2$.
\end{proof}

\begin{lem}\label{l:b67}
Proposition \ref{p:b} holds in Cases $(\mathcal{B}6)$ and $(\mathcal{B}7)$ of Table \ref{btab}. 
\end{lem}

\begin{proof}
Here $H_0 = G_2(q)$ and $T = \Omega_7(q)$ or ${\rm Sp}_{6}(q)$, according to the parity of $p$.
If $p=3$ then $G_{2}(q)$ admits an involutory graph automorphism that interchanges the two irreducible $7$-dimensional modules $L(\l_{1})$ and $L(\l_{2})$, so we only need to consider the standard embedding, labelled $(\BB6)$. Let $r$ be a prime divisor of $|\Omega|$ and $|H_0|$.

If $r=p>2$ then the proof of \cite[Lemma 2.13]{Bur6} implies that $[J_{3},J_1^{4}]$ is a derangement of order $r$. Similarly, every involution $x \in T$ with $\nu(x)=1$ is a derangement.

Finally, suppose $r \neq p$ and $r>2$. Set $i=\Phi(r,q)$ and note that $i \in \{1,2\}$. Let $x \in T$ be an element of order $r$ with $\nu(x)=2$. Let $\bar{H}=G_2$ and $\bar{G}=B_3$ (or $C_3$ if $p=2$) be the ambient simple algebraic groups over the algebraic closure $\bar{\mathbb{F}}_{q}$, and note that $x$ is contained in a maximal rank subgroup $A_2$ of $\bar{H}$. If $\bar{V}$ denotes the natural module for $\bar{G}$, then
$$\bar{V} \downarrow A_2 = \left\{\begin{array}{ll}
V_3 \oplus V_3^* & p=2 \\
V_3 \oplus V_3^* \oplus 0 & p \neq 2
\end{array}\right.$$
where $V_3$ and $0$ denote the natural and trivial $A_2$-modules, respectively. It follows that each $y \in A_2$ of order $r$ has a repeated nontrivial eigenvalue on $\bar{V}$. Since the two nontrivial eigenvalues of $x$ are distinct, we conclude that $x$ is a derangement. 
\end{proof}

\begin{lem}\label{l:brem}
Proposition \ref{p:b} holds in each of the remaining cases in Table \ref{btab}.
\end{lem}

\begin{proof}
The remaining cases are similar so we only give details in case $(\mathcal{B}13)$. Here
$T={\rm PSL}_{6}^{\e}(p)$ and $S={\rm PSU}_{4}(3)$, where $p \equiv \e\imod{6}$ and $p\geqs 5$. More precisely, $H_0=S$ or $S.2$, with $H_0 = S.2$ if and only if $p \equiv \e\imod{12}$ (see \cite[Tables 8.25 and 8.27]{BHR}). Let $r$ be a prime divisor of $|\Omega|$ and $|H_0|$, so $r \in \{2,3,5,7\}$. Let $\chi$ be the corresponding Brauer character of $6.{\rm PSU}_{4}(3)$ or $6.{\rm PSU}_{4}(3).2$ (according to the value of $p$).

If $r=p$ then $r \in \{5,7\}$ and $T$ is not $r$-elusive since $H_0$ has at most two classes of elements of order $r$. Next assume $r \ne p$ and $r>2$. Set $i=\Phi(r,p)$. If $r=3$ then $H_0$ has at most four classes of elements of order $3$, but there are at least five in $T$ (see \cite[Propositions 3.2.2 and 3.3.3]{BG}, for example). Now assume $r=5$. By inspecting $\chi$ we see that $\nu(y)=4$ for all $y \in H_0$ of order $5$, whence $T$ is $5$-elusive if and only if $i=4$ (in fact, we need $p^2 \equiv -1 \imod{25}$ so that $|\Omega|$ is divisible by $5$). Similarly, $T$ is $7$-elusive if and only if $i = 3(3+\e)/2$ (here we need the condition $p^3 \equiv -\e \imod{49}$).

Finally, let us assume $r=2$. If $H_0 = S$ then $H_0$ has a unique class of involutions, but $T$ has two such classes and thus $T$ is not $2$-elusive. Now assume that 
$H_0=S.2$, so $p \equiv \e\imod{12}$ and $T$ has three classes of involutions, with representatives labelled $t_1$, $t_2$ and $t_3$ (see \cite[Table 4.5.1]{GLS}). Note that $H_0$ is an extension of $S$ by an involutory graph automorphism of type $\gamma_1$ (see \cite[Sections 3.2.5 and 3.3.5]{BG}). By considering the Brauer character $\chi$ we deduce that the two classes of graph automorphisms in $H_0$ fuse to the $T$-classes represented by $t_1$ and $t_3$, while the involutions in $S$ are $T$-conjugate to $t_2$. We conclude that $T$ is $2$-elusive.
\end{proof}

This completes the proof of Proposition \ref{p:b}.

\section{The proof of Theorem \ref{t:main}}\label{s:c}

As in the statement of Theorem \ref{t:main}, let $G \leqs {\rm Sym}(\Omega)$ be a primitive almost simple classical group over $\mathbb{F}_{q}$ with socle $T$ and point stabiliser $H \in \mathcal{S}$. Set $H_0 = H \cap T$ and write $q=p^f$ with $p$ a prime. Let $S$ denote the socle of $H$ and let $n$ be the dimension of the natural $T$-module $V$. Let $r$ be a prime divisor of $|\Omega|$. 

If $n<6$ then \cite[Proposition 6.3.1]{BG} states that $T$ is $r$-elusive if and only if $(T,S,r)$ is one of the cases in Table \ref{t:lowdim}, so we may assume that $n \geqs 6$. Similarly, if $H \in \mathcal{A} \cup \mathcal{B}$ then the conclusion to Theorem \ref{t:main} follows from our work in Sections \ref{s:a} and \ref{s:b} (see Propositions \ref{p:a} and \ref{p:b}). In addition, Corollary \ref{p:star} implies that $T$ is $r$-elusive if all of the conditions in \eqref{e:star} hold. 

Therefore, in order to complete the proof of Theorem \ref{t:main} we may assume that $n \geqs 6$ and $H \not\in \mathcal{A} \cup \mathcal{B}$; our aim is to show that $T$ is $r$-elusive only if all the conditions in \eqref{e:star} hold. Proposition \ref{p:thm1} below is a first step towards achieving this goal. In order to state this result, recall the definition of $i$ and $c$ in \eqref{e:eqcc}, and let \eqref{e:Diamond} denote the following conditions:
\begin{equation*}
\mbox{\emph{ 
$r \neq p$, $r>2$, $r$ divides $|H_0|$ and $c > \max\{2,\sqrt{n}/2\}$.} \label{e:Diamond} \tag{$\Diamond$}}
\end{equation*}

\begin{prop}\label{p:thm1}
Let $G \leqs {\rm Sym}(\Omega)$ be a primitive almost simple classical group with socle $T$ and point stabiliser $H \in \mathcal{S}$. Let $r$ be a prime divisor of $|\Omega|$ and let $S$ denote the socle of $H$. Assume that $n \geqs 6$ and $H \not\in \mathcal{A} \cup \mathcal{B}$. Then $T$ is $r$-elusive only if the conditions in \eqref{e:Diamond} hold.
\end{prop}

\begin{proof}
We apply Theorem \ref{gursax}. For example, any element in $T$ with Jordan form $[J_{2}^2,J_1^{n-4}]$ is a derangement of order $p$. Similarly, if $p$ is odd then involutions in $T$ of type $[-I_{2},I_{n-2}]$ are also derangements.  

Now assume $r \neq p$ and $r>2$. Clearly, $T$ contains derangements of order $r$ if $|H_0|$ is indivisible by $r$, so let us assume $r$ divides $|H_0|$. If $1<c \leqs \max\{2,\sqrt{n}/2\}$ then let $x \in T$ be an element of order $r$ with $\dim C_V(x)=n-c$ (see Remark \ref{r:cc}). Here 
$\nu(x)=c$, so Theorem \ref{gursax} implies that $x$ is a derangement. Similarly, if $c=1$ then any element $x \in T$ of order $r$ with $\dim C_V(x)=n-2$ is a derangement. We conclude that $T$ is $r$-elusive only if $c > \max\{2,\sqrt{n}/2\}$.
\end{proof}

To complete the proof of Theorem \ref{t:main}, it remains to show that $T$ contains a derangement of order $r$ when the following conditions are satisfied:
\begin{equation*}
\begin{array}{c}
\multicolumn{1}{l}{\mbox{\emph{ 
$n \geqs 6$, $H \not\in \mathcal{A} \cup \mathcal{B}$, $r \neq p$, $r>2$, $r$ divides $|H_0|$ and}}} \\
\\
\max\{2,\sqrt{n}/2\} < c \leqs n/2, \\ 
\\
\multicolumn{1}{l}{\mbox{\emph{with $c<n/2$ if $T = {\rm P\Omega}_{n}^{-}(q)$.} \label{e:box} \tag{$\boxtimes$}}}
\end{array}
\end{equation*}

\begin{prop}\label{p:c}
Let $G \leqs {\rm Sym}(\Omega)$ be a primitive almost simple classical group with socle $T$ and point stabiliser $H \in \mathcal{S}$. Let $r$ be a prime divisor of $|\Omega|$ and assume that the conditions in \eqref{e:box} are satisfied. Then $T$ contains a derangement of order $r$.
\end{prop}

As in Sections \ref{p:a} and \ref{p:b}, in order to prove Proposition \ref{p:c} we will either identify a specific derangement of order $r$, or we will establish the existence of such an element by comparing  the number of $T$-classes of subgroups (or elements) in $T$ of order $r$ with the number of such $H_0$-classes in $H_0$. As before, we will write $\kappa(T,r)$ to denote the number of $T$-classes of subgroups of order $r$ in $T$ (and similarly $\kappa(H_0,r)$). Note that the conditions in \eqref{e:box} imply that $\kappa(T,r) \geqs 2$ (this quickly follows from Lemma \ref{l:floor}(iii)), so the desired conclusion follows immediately if $\kappa(H_0,r)=1$.

Before we begin the proof of Proposition \ref{p:c}, let us record a couple of useful observations. Suppose the conditions in \eqref{e:box} hold. First observe that $r \geqs 5$ since $c \geqs 3$. Also note that $r$ divides $q^{r-1}-1$ by Fermat's Little  Theorem, so $i$ divides $r-1$. In particular, if $i$ is odd then $2i$ divides $r-1$. It follows that $c$ divides $r-1$ and thus
\begin{equation}\label{r:bd}
r \geqs c+1 \geqs \left\lceil \sqrt{n}/2 \right\rceil +1.
\end{equation}

\subsection{Sporadic groups}\label{ss:spor}

We begin the proof of Proposition \ref{p:c} by considering the special case where $S$ is a sporadic group.

\begin{prop}\label{p:spor}
Proposition \ref{p:c} holds if $S$ is a sporadic group.
\end{prop}

\begin{proof}
This is a straightforward calculation, using the character table of $S$ and lower bounds on the dimensions of irreducible representations. To illustrate the general approach, we will consider the cases $S \in \{{\rm M}_{11}, {\rm J}_{2}, \mathbb{M}\}$. Set $i = \Phi(r,q)$ as in \eqref{e:phirq}.

If $S = {\rm M}_{11}$ then $r \in \{5,11\}$ and the result follows since $\kappa(H_0,r)=1$. Next suppose $S = \mathbb{M}$ is the Monster. Here $r \leqs 71$ and by inspecting the character table of $H_0=S$ (specifically, the associated power maps) we deduce that $\kappa(H_0,r) \leqs 2$. But $n \geqs 196882$ (see \cite{Jansen}) and thus 
$\kappa(T,r) \geqs \lfloor 196882/70 \rfloor-1 = 2811$
by Lemma \ref{l:floor}(iii). Now apply Corollary \ref{c:count}.

Finally, let us assume that $S = {\rm J}_{2}$, so $r \in \{5,7\}$. From the character table we see that $\kappa(S,5)=2$ and $\kappa(S,7)=1$. Therefore we may assume that $r=5$, so $i=4$ since $c \geqs 3$. If $n \geqs 13$ then $\kappa(T,5) \geqs 3$ by Lemma \ref{l:floor}(iii), so we can assume $n \leqs 12$. By inspecting \cite[Table 2]{HM2} (or \cite[Section 8.2]{BHR}), it follows that $T = {\rm PSp}_{6}(q)$ is the only possibility, and either $q=p \equiv \pm 1\imod{5}$ or $q=p^2>4$ and $p \equiv \pm 2 \imod{5}$. Clearly, neither of these conditions  on $q$ are compatible with the fact that $i=4$, so this case does not arise. (Alternatively, observe that this is the case labelled $(\mathcal{B}17)$ in Table \ref{btab}, so we can discard it since we are assuming that $H \not\in \mathcal{B}$.)

The remaining cases are very similar and we leave the reader to check the details.
\end{proof}

\subsection{Alternating groups}\label{ss:alt}

Next assume $S = A_d$ is an alternating group. Since we are assuming $H \not\in \mathcal{A}$, it follows that $V$ is not the fully deleted permutation module for $S$. Note that $r \leqs d$ since $r$ divides $|H_0|$. The following lemma gives a useful lower bound on $n$ in terms of $d$.

\begin{lem}\label{l:adbd}
If $d \geqs 15$ then $n \geqs d(d-5)/4$.
\end{lem}

\begin{proof}
First observe that $\hat{S} = 2.S$ is the full covering group of $S$. If $Z(\hat{S})$ acts nontrivially on $V$ then $p \ne 2$ and the main theorem of \cite{KT} implies that $n \geqs 2^{\lfloor (d-3)/2\rfloor}$ and the result follows. Therefore, we may assume that $S$ acts linearly on $V$, in which case the desired bound follows from \cite[Theorem 7]{James}.
\end{proof}

\begin{prop}\label{p:alter}
Proposition \ref{p:c} holds if $S$ is an alternating group.
\end{prop}

\begin{proof}
First assume $d \geqs 15$. Now $\kappa(S,r) = \lfloor d/r\rfloor$  and a combination of Lemmas \ref{l:floor}(iii) and \ref{l:adbd} implies that 
$$\kappa(T,r) \geqs \lfloor n/(r-1) \rfloor - 1 \geqs \lfloor d(d-5)/4(r-1) \rfloor - 1 > \lfloor d/r\rfloor.$$
We conclude that $T$ contains derangements of order $r$. 

Finally, let us assume that $5 \leqs d \leqs 14$. If $d \leqs 9$ then $r \in \{5,7\}$ and the result follows since $\kappa(S,r)=1$. If $d \in \{10,11,12\}$ then we may assume that $r=5$, in which case $\kappa(S,r)=2$ (if $r>5$, then $\kappa(S,r)=1$). By inspecting \cite[Table 3]{HM}, we see that $n \geqs 16$ and thus $\kappa(T,r) \geqs 3$ by Lemma \ref{l:floor}(iii). The result follows. A similar argument applies if $d \in \{13,14\}$, using the fact that $n \geqs 32$ (see \cite{HM}).
\end{proof}

\subsection{Groups of Lie type: Cross-characteristic}\label{ss:lie_cross}

Let $S$ be a simple group of Lie type over $\mathbb{F}_{t}$, where $t=\ell^e$ and $\ell \neq p$ is a prime. Set $H_0 = H \cap T$. By \eqref{r:bd} we have 
\begin{equation}\label{e:rrbb}
r \geqs \max\{5, \left\lceil \sqrt{n}/2 \right\rceil +1\}.
\end{equation}
We will make extensive use of the Landazuri-Seitz bounds in \cite{Land-S}. We consider each of the possibilities for $S$ in turn, starting with the classical groups. 

\subsubsection{Linear groups}

\begin{lem}\label{p:psl2}
Proposition \ref{p:c} holds if $S = {\rm PSL}_{2}(t)$ and $(t,p)=1$.
\end{lem}

\begin{proof}
If $t \in \{4,9\}$ then $r=5$ and $\kappa(H_0,r)=1$, so for the remainder we may assume that $t \geqs 5$ and $t \neq 9$, hence $n \geqs (t-1)/(2,t-1)$ by the main theorem of Landazuri and Seitz \cite{Land-S}. In particular, \eqref{e:rrbb} implies that
\begin{equation}\label{e:rbd1}
r \geqs \lceil \sqrt{n}/2 \rceil +1 \geqs \left\lceil \frac{1}{2}\sqrt{(t-1)/(2,t-1)} \right\rceil + 1.
\end{equation}

Suppose $x \in H_0\setminus S$ has order $r$. Then $x$ is a field automorphism and thus $r$ divides $e = \log_{\ell}t$. If $t>2^7$ then \eqref{e:rbd1} implies that $r>e$, so $t \in \{2^5, 2^7\}$ and one checks that $H_0={\rm P\Gamma L}_{2}(t)$ has a unique class of subgroups of order $r$.

For the remainder, we may assume that every element in $H_0$ of order $r$ is contained in $S$. If $r \neq \ell$ then $\kappa(S,r)=1$, so we may assume that $r=\ell$. Note that $\kappa(S,r) \leqs 2$ since $S$ has two classes of elements of order $r$. In fact, if $e=1$ then 
$\kappa(S,r)=1$ by Sylow's Theorem, so we may assume $e \geqs 2$. By Lemma \ref{l:floor}(iii) we have 
$$\kappa(T,r) \geqs \lfloor n/(\ell-1) \rfloor -1 \geqs \lfloor (\ell^e-1)/2(\ell-1) \rfloor -1.$$
This reduces us to the case $t=5^2$. Here $n \geqs 12$ and $r=5$, so $\kappa(T,r) \geqs 3$ (note that $T$ is symplectic if $n=12$ -- see \cite[Table 2(b)]{HM}).
\end{proof}

\begin{lem}\label{p:psld}
Proposition \ref{p:c} holds if $S = {\rm PSL}_{d}(t)$ and $(t,p)=1$.
\end{lem}

\begin{proof}
We may assume $d \geqs 3$. If $(d,t) = (3,2)$ or $(3,4)$ then $r \in \{5,7\}$ and $\kappa(H_0,r)=1$. In each of the remaining cases we have 
$n \geqs t^{d-1}-1$ by \cite{Land-S} and thus
\begin{equation}\label{e:rbd2}
r \geqs \lceil \sqrt{n}/2 \rceil +1 \geqs \left\lceil \frac{1}{2}\sqrt{t^{d-1}-1} \right\rceil +1
\end{equation}
In particular, $r>e$ so we only need to consider elements in ${\rm PGL}_{d}(t)$. 

Suppose $r=\ell$, so $t \geqs 5$. If $d \geqs 4$ then \eqref{e:rbd2} implies that $r>t$, so we must have $d=3$. Then $S$ has at most four conjugacy classes of elements of order $r$ (see \cite[Section 3.2.3]{BG}, for example), but Lemma \ref{l:floor}(iii) implies that $T$ has at least 
$$\lfloor n/(r-1) \rfloor -1 \geqs \lfloor (t^2-1)/(t-1) \rfloor -1 = t$$
such classes.

For the remainder, we may assume that $r \ne \ell$. Set $j = \Phi(r,t)$ (so $j$ is the smallest positive integer such that $r$ divides $t^j-1$). If $j > d/2$ then $\kappa(S,r)=1$ (see Lemma \ref{l:floor}(i)), so we may assume that $j \leqs d/2$. Now the lower bound in \eqref{e:rbd2} implies that $r>t^{(d-3)/2}$, so $j>(d-3)/2$. Therefore, either $d$ is odd and $j=(d-1)/2$, or $d$ is even and $j \in \{d/2-1, d/2\}$.

First assume $d=3$, so $j=1$ and $r$ divides $t-1$. By Lemma \ref{l:prel1}(i) we have 
$\kappa(S,r)<r$, but Lemma \ref{l:floor}(iii) implies that 
$$\kappa(T,r) \geqs \lfloor n/(r-1) \rfloor - 1 \geqs \lfloor (t^2-1)/(t-2) \rfloor - 1 \geqs t >r,$$
so $T$ contains derangements of order $r$.

Next suppose $d \geqs 5$ is odd. Here $j=(d-1)/2$, so $\lfloor d/j \rfloor = 2$ and Lemma \ref{l:floor}(ii) implies that  
$$\kappa(S,r) \leqs (r-1)/j+1 = 2(r-1)/(d-1)+1<r.$$ 
Since $r$ divides $t^{(d-1)/2}-1$, by applying Lemma \ref{l:floor}(iii) we deduce that 
$$\kappa(T,r) \geqs \lfloor n/(r-1) \rfloor - 1 \geqs \lfloor (t^{d-1}-1)/(t^{(d-1)/2}-2) \rfloor - 1 \geqs t^{(d-1)/2} >r$$
and the result follows.

Next assume $d \geqs 8$ is even. Here $\lfloor d/j \rfloor=2$ and thus 
$$\kappa(S,r) \leqs (r-1)/j+1 \leqs  2(r-1)/(d-2)+1.$$ 
Now $r$ divides $(t^j-1)/(t-1)$, so $r-1 \leqs \a$ where $\a = (t^{d/2}-1)/(t-1)-1$. Therefore, if $t>2$ then 
$$\kappa(T,r) \geqs \lfloor n/(r-1) \rfloor - 1 \geqs \lfloor (t^{d-1}-1)/\a \rfloor - 1 \geqs \a \geqs r-1 > 2(r-1)/(d-2)+1.$$
Similarly, if $t=2$ then 
$$\kappa(T,r) \geqs \lfloor (2^{d-1}-1)/\a \rfloor - 1  = 2^{d/2-1} > \frac{1}{2}(2^{d/2}-1)-\frac{1}{2} \geqs \frac{1}{2}(r-1)$$
and once again the desired result follows.

Finally, let us assume that $d \in \{4,6\}$. First assume $d=4$ so $j \in \{1,2\}$. If $j=1$ then $r \leqs t-1$ and thus $t \geqs 7$. Moreover, $\kappa(S,r) \leqs (r^2-3r+6)/2$ (see Lemma \ref{l:prel1}(ii)) and 
$$\kappa(T,r) \geqs \lfloor n/(r-1) \rfloor - 1 \geqs \lfloor (t^3-1)/(t-2) \rfloor - 1 > (t^2-5t+10)/2 \geqs (r^2-3r+6)/2.$$
Similarly, if $j=2$ then $t \geqs 4$, $\kappa(S,r) \leqs (r-1)/2+1 = (r+1)/2$ and
$$\kappa(T,r) \geqs \lfloor n/(r-1) \rfloor - 1 \geqs \lfloor (t^3-1)/t \rfloor - 1 > (t+2)/2 \geqs (r+1)/2.$$

Now assume $d=6$, so $j \in \{2,3\}$. If $j=2$ then $t \geqs 4$ and $\kappa(S,r) \leqs (r^2+15)/8$ by Lemma \ref{l:prel1}(iii), whereas Lemma \ref{l:floor}(iii) implies that 
$$\kappa(T,r) \geqs \lfloor n/(r-1) \rfloor - 1 \geqs \lfloor (t^5-1)/t \rfloor - 1 > (r^2+15)/8.$$
Finally, if $j = 3$ then $r \leqs (t^3-1)/(t-1) = t^2+t+1$ and $\kappa(S,r) \leqs (r-1)/3+1 = (r+2)/3$. However,
$$\kappa(T,r) \geqs \lfloor n/(r-1) \rfloor - 1 \geqs \lfloor (t^5-1)/t(t+1) \rfloor - 1 > (r+2)/3$$
and the desired result follows.
\end{proof}

\subsubsection{Unitary groups} 

\begin{lem}\label{l:psud}
Proposition \ref{p:c} holds if $S = {\rm PSU}_{d}(t)$ and $(t,p)=1$.
\end{lem}

\begin{proof}
In view of Lemma \ref{p:psl2}, we may assume that $d \geqs 3$. If $d=4$ and $t \leqs 3$ then $r \in \{5,7\}$ and $\kappa(H_0,r)=1$, so we may assume  that $t>3$ if $d=4$. Therefore, \cite{Land-S} implies that 
\begin{equation}\label{e:unb}
n \geqs \left\{\begin{array}{ll}
t(t^{d-1}-1)/(t+1) & \mbox{$d$ odd} \\
(t^{d}-1)/(t+1) & \mbox{$d$ even}
\end{array}\right.
\end{equation}
and thus
$$r \geqs \lceil \sqrt{n}/2 \rceil +1 \geqs \left\lceil \frac{1}{2}\sqrt{t(t^{d-1}-1)/(t+1)} \right\rceil +1>\log_{\ell}t.$$
Therefore, every element of order $r$ in $H_0$ is contained in ${\rm PGU}_{d}(t)$. In fact, the same bound implies that $r>t$ if $d \geqs 4$, so $r=\ell$ only if $d=3$. 

Suppose $r=\ell$, so $S = {\rm PSU}_{3}(t)$ and $t \geqs 5$. Now $\kappa(S,r) \leqs 4$ and by combining the lower bound on $n$ in \eqref{e:unb} with Lemma \ref{l:floor}(iii), we see that $\kappa(T,r) \geqs t-1$. Therefore, we may assume that $t=5$ and thus $n \geqs 20$. If $n=20$ then $T$ is a symplectic group (see \cite[Table 2]{HM2}), so Lemma \ref{l:floor}(iii) implies that $\kappa(T,r) \geqs 5$ and the result follows.

For the remainder we may assume $r \ne \ell$. Set $j = \Phi(r,t)$ and 
$$c' = \left\{\begin{array}{ll}
j & j \equiv 0 \imod{4} \\
j/2 & j \equiv 2 \imod{4}  \\
2j & \mbox{$j$ odd.}
\end{array}\right.$$
Note that $\kappa(S,r)=1$ if $c'>d/2$ (see Lemma \ref{l:floor}(i)), so we may assume that $c' \leqs d/2$.

First consider the special case $t=2$, so $d \geqs 4$ (since ${\rm PSU}_{3}(2)$ is not simple). The cases with $d \leqs 9$ can be checked directly. For example, suppose $S = {\rm PSU}_{9}(2)$.
If $r>5$ then $c' \geqs 5$ and thus $\kappa(S,r)=1$. If $r=5$ then $\kappa(S,r)=2$ and \eqref{e:unb} implies that $n \geqs 170$, so $\kappa(T,r) \geqs \lfloor 170/4\rfloor-1=41$.
Now assume $d \geqs 10$, so
$$r \geqs \left\lceil \frac{1}{2}\sqrt{2(2^{d-1}-1)/3} \right\rceil +1> 2^{(d-4)/2}+1$$
and thus $j>(d-4)/2$. If $j \equiv 0 \imod{4}$ then $r$ divides $2^{j/2}+1$ and $d-3 \leqs j=c' \leqs d/2$, which is absurd since $d \geqs 10$. Similarly, if $j$ is odd then $(d-4)/2<j\leqs d/4$ and once again we reach a contradiction. Finally, suppose $j \equiv 2 \imod{4}$. Here $d-3\leqs j \leqs d$ and $c'=j/2$, so $\lfloor d/c' \rfloor=2$ since $d \geqs 10$. Therefore, $\kappa(S,r) \leqs (r-1)/c'+1$ by Lemma \ref{l:floor}(ii), whereas $T$ has at least
$$\lfloor n/(r-1) \rfloor - 1 \geqs \left\lfloor \frac{2(2^{d-1}-1)/3}{2^{d/2}+1} \right\rfloor - 1 >\frac{1}{d}2^{d/2+1}+1 \geqs \frac{1}{j}2^{j/2+1}+1 \geqs \frac{r-1}{c'}+1$$
classes of subgroups of order $r$. The result follows.

Finally, let us assume $t \geqs 3$. First observe that
$$r \geqs \left\lceil \frac{1}{2}\sqrt{t(t^{d-1}-1)/(t+1)} \right\rceil +1 > t^{(d-3)/2}+1.$$
If $j \equiv 0 \imod{4}$ then $4 \leqs j=c' \leqs d/2$, so $d \geqs 8$. Moreover, $r$ divides $t^{j/2}+1$ and thus $d-2 \leqs j = c'  \leqs d/2$, but this is incompatible with the bound $d \geqs 8$. Next assume $j$ is odd, so $2 \leqs 2j=c' \leqs d/2$ and $d \geqs 4$. Since 
$(d-3)/2<j \leqs d/4$, it follows that $d \in \{4,5\}$ and $j=1$, so $r \leqs t-1$. In particular, $\kappa(S,r) \leqs (r+1)/2$ (see Lemma \ref{l:floor}(ii)) and
$$\kappa(T,r) \geqs \lfloor n/(r-1) \rfloor-1 \geqs \left\lfloor \frac{(t^4-1)/(t+1)}{t-2}\right\rfloor-1>t/2 \geqs (r+1)/2,$$
so the result follows.

To complete the proof of the lemma, we may assume that $t \geqs 3$ and $j \equiv 2 \imod{4}$, in which case $t^{(d-3)/2}+1 < r \leqs t^{j/2}+1$ and thus $d-2 \leqs j \leqs d$. In particular, $d \not\equiv 1 \imod{4}$. For now, we will assume that $d \geqs 6$, so $\lfloor d/c'\rfloor=2$ and Lemma \ref{l:floor}(ii) implies that 
$$\kappa(S,r) \leqs (r-1)/c'+1 = 2(r-1)/j+1.$$

If $d \equiv 0 \imod{4}$ and $d \geqs 8$ then $j=d-2$ and 
$$\kappa(T,r) \geqs \lfloor n/(r-1) \rfloor-1 \geqs \left\lfloor \frac{(t^d-1)/(t+1)}{t^{d/2-1}}\right\rfloor-1>\frac{2t^{d/2-1}}{d-2}+1 \geqs \frac{2(r-1)}{j}+1.$$
Similarly, if $d \equiv 3 \imod{4}$ and $d \geqs 7$, then $j=d-1$ and we see that $\kappa(T,r) > 2(r-1)/j+1$. Now assume $d \equiv 2 \imod{4}$, so $d \geqs 6$ and $j=d$. Here $r$ divides $(t^{d/2}+1)/(t+1)$, so $r-1 \leqs (t^{d/2}+1)/(t+1)-1=\a$ and thus  
$$\kappa(T,r) \geqs \lfloor n/(r-1) \rfloor-1 \geqs \left\lfloor \frac{(t^d-1)/(t+1)}{\a}\right\rfloor-1>\frac{2}{d}\left(\frac{t^{d/2}+1}{t+1}-1\right)+1 \geqs \frac{2(r-1)}{d}+1.$$

Therefore, to complete the proof we may assume that $d \in \{3,4\}$, in which case $j=2$ and $r$ divides $t+1$, so $t \geqs 4$. If $d=4$ then $\kappa(S,r) \leqs (r^2-3r+6)/2$ (see Lemma \ref{l:prel1}(ii)) and
$$\kappa(T,r) \geqs \lfloor n/(r-1) \rfloor-1 \geqs \left\lfloor \frac{t^4-1}{(t+1)t} \right\rfloor-1 > \frac{1}{2}(t^2-t+4) \geqs \frac{1}{2}(r^2-3r+6).$$

Finally suppose $d=3$, so $\kappa(S,r) <r$ by Lemma \ref{l:prel1}(i). If $r<t+1$ then $r \leqs (t+1)/2$ and by applying Lemma \ref{l:floor}(iii) we deduce that $\kappa(T,r) \geqs 2t-1 \geqs r$. Therefore, we may assume that $r=t+1$, so $t \geqs 4$ is even. If $S = {\rm PSU}_{3}(4)$ then $r=5$ and $\kappa(S,r)=2$, whereas $\kappa(T,r) \geqs 3$ since $n \geqs 12$ (note that $T$ is a symplectic group if $n=12$; see \cite[Table 2]{HM2}). Now assume $t \geqs 16$ and let $\omega \in \mathbb{F}_{t^2}$ be a primitive $r$-th root of unity. As noted in the proof of Lemma \ref{l:prel1}(i), any subgroup of ${\rm PGU}_{3}(t)$ of order $r$ is conjugate to a subgroup of the form $\la [1,1, \omega]Z \ra$ or $\la [1, \omega, \omega^k]Z \ra$ for some $1< k < r$, where $Z$ denotes the centre of ${\rm GU}_{3}(t)$. In fact, we can exclude $k \in \{2,4,8\}$ since 
$$[1,\omega, \omega^2]^{(t+2)/2} \sim [1,\omega, \omega^{(t+2)/2}],\;\; [1,\omega, \omega^4]^{(3t+4)/4} \sim [1,\omega, \omega^{(3t+4)/4}]$$
and
$$[1,\omega, \omega^8]^{(7t+8)/8} \sim [1,\omega, \omega^{(7t+8)/8}],$$
where $\sim$ denotes ${\rm GU}_{3}(t)$-conjugacy. Therefore, $\kappa(S,r) \leqs t-3$ and Lemma \ref{l:floor}(iii) implies that $\kappa(T,r) \geqs t-2$. This completes the proof of the lemma. 
\end{proof}

\subsubsection{Symplectic groups}

\begin{lem}\label{p:psp4}
Proposition \ref{p:c} holds if $S = {\rm PSp}_{4}(t)'$ and $(t,p)=1$.
\end{lem}

\begin{proof}
If $t \in \{2,3\}$ then $r=5$ and $\kappa(S,r)=1$, so for the remainder we may assume that $t \geqs 4$. By \cite{Land-S} we have
\begin{equation}\label{e:nbd}
n \geqs \left\{\begin{array}{ll}
\frac{1}{2}(t^{2}-1) & \mbox{$t$ odd} \\
\frac{1}{2}t(t-1)^2 & \mbox{$t$ even}
\end{array}\right.
\end{equation}

Suppose $t=4$, so $r \in \{5,17\}$ and $n \geqs 18$. Now $\kappa(S,5)=3$ and $\kappa(S,17)=1$, so the result follows from the lower bound on $\kappa(T,r)$ in Lemma \ref{l:floor}(iii). Next assume $t=5$, so $r \in \{5,13\}$ and $n \geqs 12$. Since $\kappa(S,5)=4$ and $\kappa(S,13)=1$, we may assume $r=5$. If $n > 13$ then by inspecting \cite[Table 2]{HM2} we deduce that $n \geqs 40$ and thus $\kappa(T,5) \geqs 5$. Therefore, we may assume that $n \in \{12,13\}$. By considering the corresponding Frobenius-Schur indicator in \cite[Table 2]{HM2} we see that $T = {\rm PSp}_{12}(q)$ or $\Omega_{13}(q)$. Set $i = \Phi(r,q)$ as before and note that $i \in \{1,2,4\}$. In fact, $i=4$ is the only possibility since $c \geqs 3$, so $q^2 \equiv -1 \imod{5}$. However, by inspecting the irrationalities of the corresponding Brauer character in \cite[Table 2]{HM2}, we see that $q^2 \equiv 1 \imod{5}$, which is a contradiction.

For the remainder we may assume that $t \geqs 7$, in which case \eqref{r:bd} implies that
$$r \geqs \left\lceil \frac{1}{2}\sqrt{(t^2-1)/2}\right\rceil+1 > \log_{\ell}t$$
and thus every element in $H_0$ of order $r$ is contained in $S$. 

First assume $r=\ell$, so $t$ is odd. According to \cite[Proposition 3.4.10]{BG}, $S$ has six classes of elements of order $r$, and Lemma \ref{l:floor}(iii) implies that 
$$\kappa(T,r) \geqs \lfloor n/(t-1) \rfloor - 1 \geqs \lfloor (t+1)/2 \rfloor -1 = (t-1)/2.$$
Therefore, we may assume that $t \in \{7,11,13\}$. In each of these cases one checks that $\kappa(S,r)=4$, so we can assume $r=t=7$ and $c \in \{3,6\}$. Note that $n \geqs 24$. If $c=3$ then $\kappa(T,r) \geqs \lfloor 24/3 \rfloor-1=7$, so we can assume $c=6$. If $n>25$ then $n \geqs 126$ (see \cite[Table 2]{HM2}) and the desired result follows, so let us assume that $n \in \{24,25\}$. 
Suppose $x \in S$ has order $7$ and let $\chi$ be the corresponding Brauer character. Since $c=6$, each nontrivial $7$-th root of unity occurs as an eigenvalue of $x$ with equal multiplicity (in terms of the action of $x$ on $\bar{V} = V \otimes K$, where $K$ is the algebraic closure of $\mathbb{F}_{q}$), so $\chi(x)$ is an integer. However, \cite[Table 2]{HM2} indicates that $\chi$ has a $b7$ irrationality in the standard Atlas notation, which means that $\chi(x)$ is not an integer for some $x \in S$ of order $7$. Therefore, $c \ne 6$ when $n \in \{24,25\}$ and the result follows. 
 
Now assume $r \ne \ell$. Set $j = \Phi(r,t)$ and note that $j \in \{1,2,4\}$. If $j=4$ then $\kappa(S,r)=1$, so we may assume that $j \in \{1,2\}$, in which case $r \leqs t+1$ and 
$\kappa(S,r) \leqs (r+1)/2$ (see Lemma \ref{l:floor}(ii)). If $t$ is even then a combination of \eqref{e:nbd} and Lemma \ref{l:floor}(iii) implies that $\kappa(T,r)> t/2+1$. Similarly, if $t$ is odd then $r \leqs (t+1)/2$ and the same conclusion holds.
\end{proof}

\begin{lem}\label{p:psp6}
Proposition \ref{p:c} holds if $S = {\rm PSp}_{6}(t)$ and $(t,p)=1$. 
\end{lem}

\begin{proof}
If $t=2$ then $r \in \{5,7\}$ and $\kappa(S,r)=1$. For the remainder we may assume that $t \geqs 3$, in which case \cite{Land-S} gives
\begin{equation}\label{e:nbd2}
n \geqs \left\{\begin{array}{ll}
\frac{1}{2}(t^{3}-1) & \mbox{$t$ odd} \\
\frac{1}{2}t^2(t^2-1)(t-1) & \mbox{$t$ even}
\end{array}\right.
\end{equation}
In particular, $r \geqs \lceil \frac{1}{2}\sqrt{(t^3-1)/2} \rceil+1>\log_{\ell}t$. In fact, the same bound implies that $r>t$ if $t \geqs 7$, so if $r=\ell$ then $t=5$ is the only possibility. Now $S = {\rm PSp}_{6}(5)$ has $13$ classes of elements of order $5$, but $T$ has at least $\lfloor 62/(5-1) \rfloor = 15$ since $n \geqs 62$. 

Now assume $r \ne \ell$. Set $j = \Phi(r,t)$ and note that $j \in \{1,2,3,4,6\}$. If $j>2$ then $\kappa(S,r)=1$ and the result follows. Now assume that $j \in \{1,2\}$, so $r \leqs t+1$. By arguing as in the proof of Lemma \ref{l:prel1}(iii) we see that $\kappa(S,r) \leqs (r^2+15)/8$ and in the usual way, via \eqref{e:nbd2} and Lemma \ref{l:floor}, it is easy to check that $\kappa(T,r)>\kappa(S,r)$.
\end{proof}

\begin{lem}\label{p:pspd}
Proposition \ref{p:c} holds if $S = {\rm PSp}_{d}(t)'$ and $(t,p)=1$. 
\end{lem}

\begin{proof}
We may assume $d \geqs 8$. By \cite{Land-S} we have 
\begin{equation}\label{e:nbd3}
n \geqs \left\{\begin{array}{ll}
\frac{1}{2}(t^{d/2}-1) & \mbox{$t$ odd} \\
\frac{1}{2}t^{d/2-1}(t^{d/2-1}-1)(t-1) & \mbox{$t$ even}
\end{array}\right.
\end{equation}
and \eqref{r:bd} implies that $r>t^{(d-4)/4}+1$. In particular, $r \neq \ell$ and every element in $H_0$ of order $r$ is contained in $S$. Let $j = \Phi(r,t)$ and set $c'=2j$ if $j$ is odd, and $c'=j$ if $j$ is even.

If $j$ is odd and $j>d/4$ then $c'>d/2$ and thus $\kappa(S,r)=1$. The same conclusion holds if $j$ is even and $j>d/2$. Note that if $j$ is even then $r$ divides $t^{j/2}+1$, so the bound  $r>t^{(d-4)/4}+1$ implies that $j>(d-4)/2$. Therefore, we may assume that one of the following holds:
\begin{itemize}\addtolength{\itemsep}{0.2\baselineskip}
\item[(a)] $j$ odd: Either $d \equiv 4 \imod{8}$ and $j=d/4$, or $d \equiv 6 \imod{8}$ and $j=(d-2)/4$.
\item[(b)] $j$ even: Either $d \equiv 0 \imod{4}$ and $j=d/2$, or $d \equiv 2 \imod{4}$ and $j=(d-2)/2$.
\end{itemize}

First assume that $j$ is odd and $d \equiv 4 \imod{8}$, so $d \geqs 12$,  $c'=d/2$ and $r$ divides $t^{d/4}-1$. Since $\lfloor d/c' \rfloor = 2$, Lemma \ref{l:floor}(ii) implies that $\kappa(S,r) \leqs 2(r-1)/d+1$, whereas $T$ has at least 
$$\lfloor n/(r-1) \rfloor -1 \geqs \left\lfloor \frac{(t^{d/2}-1)/2}{t^{d/4}-2}\right\rfloor - 1 \geqs \frac{1}{2}t^{d/4}-1 \geqs \frac{1}{2}(r-1) > 2(r-1)/d+1$$
such classes. A similar argument applies if $j$ is odd and $d \equiv 6 \imod{8}$.

Next assume $j$ is even and $d \equiv 2 \imod{4}$, so $d \geqs 10$, $c'=(d-2)/2$ and $r$ divides $t^{(d-2)/4}+1$. In addition, since $\lfloor d/c' \rfloor=2$ we have $\kappa(S,r) \leqs 2(r-1)/(d-2)+1$. By applying Lemma \ref{l:floor}(iii) and the lower bound on $n$ given in \eqref{e:nbd3}, it is easy to check that $\kappa(T,r) \geqs t^{(d-2)/4}$. The desired result follows since $t^{(d-2)/4} \geqs r-1 > 2(r-1)/(d-2)+1$.

Finally, suppose that $j$ is even and $d \equiv 0 \imod{4}$, in which case $d \geqs 8$, $c'=d/2$ and $r$ divides $t^{d/4}+1$. In particular, $\kappa(S,r) \leqs 2(r-1)/d+1$. We claim that $\kappa(T,r) \geqs t^{d/4}$, which is sufficient since $t^{d/4} \geqs r-1>2(r-1)/d+1$. If $t$ is even, this follows in the usual way via Lemma \ref{l:floor}(iii) and \eqref{e:nbd3}. If $t$ is odd then $t^{d/4}+1$ is even and thus $r$ divides $(t^{d/4}+1)/2$, whence Lemma \ref{l:floor}(iii) implies that
$$\kappa(T,r) \geqs \lfloor n/(r-1) \rfloor - 1 \geqs \left\lfloor \frac{(t^{d/2}-1)/2}{(t^{d/4}-1)/2} \right\rfloor - 1  = t^{d/4}$$
as claimed. 
\end{proof}

\subsubsection{Orthogonal groups}

\begin{lem}\label{p:orthog}
Proposition \ref{p:c} holds if $S = {\rm P\Omega}_{d}^{\e}(t)$ and $(t,p)=1$.
\end{lem}

\begin{proof}
We may assume that $d \geqs 7$, with $t$ odd if $d$ is odd. If $S=\Omega_7(3)$ then $r \in \{5,7,13\}$ and $\kappa(S,r)=1$. Next assume $S = \Omega_{8}^{+}(2)$. Here $r \in \{5,7\}$ with $\kappa(S,5)=3$ and $\kappa(S,7)=1$. By \cite{Land-S} we have $n \geqs 8$. If $n=8$ then $T = {\rm P\Omega}_{8}^{+}(p)$ and this is the case labelled $(\mathcal{B}4)$ in Table \ref{btab}. Therefore, we may assume that $n>8$, in which case $n \geqs 28$ (see \cite[Table 2]{HM2}) and thus Lemma \ref{l:floor}(iii) implies that $\kappa(T,5) \geqs 6$.

In each of the remaining cases, the Landazuri-Seitz \cite{Land-S} bounds imply that
$$n \geqs \left\{\begin{array}{ll}
(t^{(d-2)/2}+1)(t^{(d-4)/2}-1) & \mbox{$d$ even} \\
t^{(d-3)/2}(t^{(d-3)/2}-1) & \mbox{$d$ odd.}
\end{array}\right.$$
Suppose $d$ is odd. Then
$$r \geqs \lceil \sqrt{n}/2 \rceil +1 \geqs \left\lceil \frac{1}{2}\sqrt{t^{(d-3)/2}(t^{(d-3)/2}-1)}\right\rceil +1>t^{(d-1)/4}+1$$
and thus $r \ne \ell$ and every element in $H_0$ of order $r$ is contained in $S$. Set $j = \Phi(r,t)$ and $c'=2j$ if $j$ is odd, otherwise $c'=j$. Note that the above lower bound on $r$ implies that $j>(d-1)/4$. If $j$ is odd then $c'>(d-1)/2$ and thus $c' \geqs (d+1)/2$ and $\kappa(S,r)=1$. Similarly, if $j$ is even then $r$ divides $t^{j/2}+1$, so $c'>(d-1)/2$ and the same conclusion holds.

Finally, let us assume $d$ is even. The cases with $d=8$ and $t \in \{2,3\}$ can be checked directly. For example, if $S = {\rm P\Omega}_{8}^{+}(3)$ then $r \in \{5,7,13\}$. If $r \in \{7,13\}$ then $c'=6$ and thus $\kappa(S,r)=1$. If $r=5$ then $c'=4$ and $\kappa(S,r)=2$, but $n \geqs 224$ and thus $\kappa(T,r)>2$. In all of the remaining cases we have 
$$r \geqs \lceil \sqrt{n}/2 \rceil +1 \geqs \left\lceil \frac{1}{2}\sqrt{(t^{(d-2)/2}+1)(t^{(d-4)/2}-1)}\right\rceil +1>t^{d/4}+1$$
and by arguing as in the $d$ odd case we deduce that $\kappa(S,r)=1$.
\end{proof}

\subsubsection{Exceptional groups}

\begin{lem}\label{p:e78cross}
Proposition \ref{p:c} holds if $S \in \{E_7(t), E_8(t)\}$ and $(t,p)=1$.
\end{lem}

\begin{proof}
First assume $S = E_8(t)$. Here $n \geqs t^{27}(t^2-1)$ by \cite{Land-S}, so \eqref{r:bd} implies that
$$r \geqs \lceil \sqrt{n}/2\rceil +1 \geqs \left\lceil \frac{1}{2}\sqrt{t^{27}(t^2-1)} \right\rceil+1>t^{13}.$$
Therefore $r$ divides $|S|$, $r \ne \ell$ and by considering the order of $S$ we deduce that  
$j=\Phi(r,t) \in \{14,18,20,24,30\}$. Hence $r$ divides $t^{j/2}+1$, so $j=30$ is the only possibility. However, if $j=30$ then $r$ divides $(t^{15}+1)/(t^5+1) = t^{10}-t^5+1$, which is incompatible with the bound $r>t^{13}$. The case $S = E_7(t)$ is entirely similar, using the bound $n \geqs t^{15}(t^2-1)$ from \cite{Land-S}. 
\end{proof}

\begin{lem}\label{p:e6cross}
Proposition \ref{p:c} holds if $S \in \{ E_6(t), {}^2E_6(t) \}$ and $(t,p)=1$.
\end{lem}

\begin{proof}
Here $n \geqs t^{9}(t^2-1)$ (see \cite[Table 5.3.A]{KL}) and we deduce that $r>t^4+1$. Set $j=\Phi(r,t)$ and first assume $S = E_6(t)$. Since $r>t^4+1$ and $r$ divides $|S|$, it follows that $j \in \{9,12\}$. If $j=12$ then $r$ divides $(t^6+1)/(t^2+1) = t^4-t^2+1$, which contradicts the bound $r>t^4+1$. Now assume $j=9$, in which case $r$ divides $(t^9-1)/(t^3-1) = t^6+t^3+1$. By inspecting the structure of the maximal tori in $S$ (see \cite[Section 2.7]{KS}, for example), it follows that every subgroup of $S$ of order $r$ is contained in a cyclic maximal torus of order $t^6+t^3+1$. Since $S$ has a unique class of such maximal tori, we deduce that $\kappa(S,r)=1$ and the result follows.

A very similar argument applies if $S = {}^2E_6(t)$. Here $j \in \{10,12,18\}$ and we can rule out $j=12$ as above. Similarly, if $j=10$ then $r$ divides $(t^5+1)/(t+1)$, but this is not possible since $r>t^4+1$. Finally, if $j=18$ then $r$ divides $t^6-t^3+1$, which implies that every subgroup of $S$ of order $r$ is contained in a cyclic maximal torus of order $t^6-t^3+1$. Once again we deduce that $\kappa(S,r)=1$ since $S$ has a unique class of such tori.
\end{proof}

\begin{lem}\label{p:f4cross}
Proposition \ref{p:c} holds if $S \in \{ F_4(t), {}^2F_4(t)' \}$ and $(t,p)=1$. 
\end{lem}

\begin{proof}
First assume that $S=F_4(2)$, so $r \in \{5,7,13,17\}$ and $n \geqs 44$ (see \cite{Land-S}). If $r \in \{5,13,17\}$ then the character table of $S$ indicates that $\kappa(S,r)=1$, so we may assume $r=7$. Now $S$ has two classes of subgroups of order $7$, but Lemma \ref{l:floor}(iii) implies that $T$ has at least $\lfloor 44/6 \rfloor-1=6$ such classes. The result follows.

Next suppose that $S=F_4(t)$ and $t \geqs 3$. Here $n \geqs t^6(t^2-1)$ (see \cite[Table 5.3.A]{KL}) and thus
$$r \geqs \left\lceil \frac{1}{2}\sqrt{t^6(t^2-1)} \right\rceil+1>t^3+1.$$
Set $j=\Phi(r,t)$. Since $r$ divides $|S|$ and $r>t^3+1$, it follows that $j \in \{8,12\}$. If $j=12$ then $r$ divides $(t^6+1)/(t^2+1) = t^4-t^2+1$ and we deduce that every subgroup of $S$ of order $r$ is contained in a cyclic maximal torus of order $t^4-t^2+1$. But there is a unique conjugacy class of such tori, whence $\kappa(S,r)=1$. An entirely similar argument applies if $j=8$, using the fact that $S$ has a unique class of cyclic maximal tori of order $t^4+1$.

Now assume $S = {}^2F_4(t)'$, so $t=2^{2m+1}$ with $m \geqs 0$. If $t=2$ then $r \in \{5,13\}$ and by inspecting the character table of $S$ we deduce that $\kappa(S,r)=1$. Now assume $t \geqs 8$. Here \cite{Land-S} gives $n \geqs t^4(t-1)\sqrt{t/2}$, which implies that $r>t^2+1$. Set $j = \Phi(r,t)$ and observe that $j \in \{6,12\}$. If $j=6$ then $r$ divides $(t^3+1)/(t+1) = t^2-t+1$, which contradicts the bound $r>t^2+1$. Now assume $j=12$, in which case $r$ divides 
$$\frac{t^6+1}{t^2+1} = (t^2 - \sqrt{2t^3} + t - \sqrt{2t} +1)(t^2 + \sqrt{2t^3} + t + \sqrt{2t} +1).$$
Since $r>t^2+1$ it follows that $r$ divides $t^2 + \sqrt{2t^3} + t + \sqrt{2t} +1$ and we deduce that every subgroup of $S$ of order $r$ is contained in a cyclic maximal torus in $S$ of order  
$t^2 + \sqrt{2t^3} + t + \sqrt{2t} +1$. The result now follows since $S$ has a unique conjugacy class of such tori and thus $\kappa(S,r)=1$.
\end{proof}

\begin{lem}\label{p:g2cross}
Proposition \ref{p:c} holds if $S \in \{ {}^2B_2(t), G_2(t)', {}^2G_2(t)' \}$ and $(t,p)=1$.
\end{lem}

\begin{proof}
First assume $S = {}^2B_2(t)$, so $t=2^{2m+1}$ and $m \geqs 1$.  If $t=8$ then $r \in \{5,7,13\}$ and $\kappa(S,r)=1$. Now assume $t \geqs 32$, so $n \geqs (t-1)\sqrt{t/2}$ by \cite{Land-S}. Therefore $r>\log_2t$ by \eqref{r:bd}, so every element in $H_0$ of order $r$ is contained in $S$. The maximal tori of $S$ are cyclic of order $t-1$, $t+\sqrt{2t}+1$ and $t-\sqrt{2t}+1$, and $S$ has a unique class of tori of each order. Since these  orders are pairwise coprime, it follows that $\kappa(S,r)=1$ and the result follows.

Next assume that $S = G_2(t)'$. If $t \in \{2,3\}$ then $r \in \{7,13\}$ and $\kappa(S,r)=1$. If $t=4$ then $n \geqs 12$ (see \cite[Table 5.3.A]{KL}) and $r \in \{5,7,13\}$. We have $\kappa(S,r)=1$ if $r \in \{7,13\}$, so we may assume $r=5$ and thus $\kappa(S,r)=2$. The  result now follows from Lemma \ref{l:floor}(iii) since $T$ is a symplectic group when $n=12$ (see \cite[Table 2]{HM2}). Now assume $t \geqs 5$. Here \cite{Land-S} gives $n \geqs t(t^2-1)$ and thus
$$r \geqs \left\lceil \frac{1}{2}\sqrt{t(t^2-1)}\right\rceil+1>t+1.$$
Set $j = \Phi(r,t)$ and note that $j \in \{3,6\}$ since $r$ divides $|S|$ and $r>t+1$. If $j=3$ then $r$ divides $t^2+t+1$ and we deduce that every subgroup of $S$ of order $r$ is contained in a cyclic maximal torus of order $t^2+t+1$. Since $S$ has a unique class of such tori, we see that $\kappa(S,r)=1$. An entirely similar argument applies if $j=6$.

Finally, let us assume $S = {}^2G_2(t)'$, where $t=3^{2m+1}$ and $m \geqs 0$. If $t=3$ then $r=7$ and $\kappa(S,r)=1$. Now assume $t \geqs 27$. By \cite{Land-S} we have $n \geqs t(t-1)$ and thus
$$r \geqs \left\lceil \frac{1}{2}\sqrt{t(t-1)}\right\rceil+1>\frac{1}{2}(t+1)$$
so $r>\log_3t$. Set $j = \Phi(r,t)$ and observe that $j \in \{1,2,6\}$. If $j=1$ then $r$ divides $(t-1)/2$, which is incompatible with the bound $r>(t+1)/2$. Similarly, if $j=2$ then $r$ divides $(t+1)/2$ and once again we have reached a contradiction. Finally, suppose that $j=6$, in which case $r$ divides 
$$\frac{t^3+1}{t+1} = t^2-t+1 = (t+\sqrt{3t}+1)(t-\sqrt{3t}+1).$$
If $r$ divides $t+\sqrt{3t}+1$ then every subgroup of $S$ of order $r$ is contained in a cyclic maximal torus of order $t+\sqrt{3t}+1$; there is a unique class of such tori, so $\kappa(S,r)=1$. A very similar argument applies if $r$ divides $t-\sqrt{3t}+1$.
\end{proof}

\begin{lem}\label{p:3d4cross}
Proposition \ref{p:c} holds if $S = {}^3D_4(t)$ and $(t,p)=1$. 
\end{lem}

\begin{proof}
Here \cite{Land-S} gives $n \geqs t^3(t^2-1)$. First assume $t=2$, so $r \in \{7,13\}$ and $n \geqs 24$. Since $\kappa(S,7)=2$ and $\kappa(S,13)=1$, the result follows from Lemma \ref{l:floor}(iii). 

Next assume $t=3$, so $r \in \{7,13,73\}$ and $n \geqs 216$. If $r=73$ then $\kappa(S,r)=1$ by Sylow's Theorem. Similarly, if $r \in \{7,13\}$ then the Sylow $r$-subgroups of $S$ are isomorphic to $Z_r \times Z_r$, which implies that $\kappa(S,r) \leqs r+1$. But Lemma \ref{l:floor}(iii) implies that $T$ has at least $\lfloor 216/12\rfloor-1 = 17$ such classes, so the result follows.
The case $t=4$ is entirely similar (here $r \in \{5,7,13,241\}$ and $n \geqs 960$).

To complete the proof of the lemma, we may assume that $t \geqs 5$. First observe that
$$r \geqs \left\lceil \frac{1}{2}\sqrt{t^3(t^2-1)}\right\rceil+1>t^2+1.$$
In particular, $j = \Phi(r,t) \in \{3,6,12\}$. If $j=6$ then $r$ divides $(t^3+1)/(t+1) = t^2-t+1$, but this contradicts the bound $r>t^2+1$. Next suppose that $j=12$, so $r$ divides $(t^6+1)/(t^2+1) = t^4-t^2+1$. Every subgroup of $S$ of order $r$ is contained in a cyclic maximal torus of order $t^4-t^2+1$; since $S$ has a unique class of such tori, it follows that $\kappa(S,r)=1$.

Finally, let us assume that $j=3$, so $r$ divides $(t^3-1)/(t-1) = t^2+t+1$. If $t \geqs 7$ then the bound $r \geqs \left\lceil \sqrt{t^3(t^2-1)}/2 \right\rceil+1$ implies that $r>t^2+t+1$, so we may assume that $t=5$ and thus $r=31$. The Sylow $31$-subgroups of $S$ are isomorphic to $Z_{31} \times Z_{31}$, so $\kappa(S,31) \leqs 32$. But $n \geqs 3000$ so Lemma \ref{l:floor}(iii) implies that $\kappa(T,31) \geqs 99$.
\end{proof}

This completes the proof of Proposition \ref{p:c} in the case where $S$ is a simple group of Lie type in non-defining characteristic.

\subsection{Groups of Lie type: Defining characteristic}\label{ss:lie_def}

In this final section we complete the proof of Proposition \ref{p:c} by considering the case where $S$ is a simple group of Lie type over $\mathbb{F}_{p^e}$, for some positive integer $e$. 

Let $K$ be the algebraic closure of $\mathbb{F}_{p}$, let $M$ be a $K\hat{S}$-module affording a representation $\rho$ and let $\gamma$ be an automorphism of $\hat{S}$. Following \cite[p.192]{KL}, we write $M^{\gamma}$ for the space $M$ with $\hat{S}$-action given by the representation $\gamma\rho$ (acting on the right) and we say that the $K\hat{S}$-modules $M$ and $M^{\gamma}$ are quasiequivalent. In particular, if $\gamma$ is a field automorphism of $\hat{S}$ induced by the map $\lambda \mapsto \lambda^p$ on $K$ then we will write $M^{\gamma^z} = M^{(z)}$ for all $z \in \mathbb{N}$. 

By a theorem of Steinberg \cite{St}, the irreducible $K\hat{S}$-modules are parameterised by an appropriate set of weights for the ambient simple algebraic group $\bar{S}$ over $K$, with respect to a fixed set of fundamental dominant weights. We will write $\{\l_1, \ldots, \l_k\}$ for the latter weights, where we adopt the standard labelling given in Bourbaki \cite{Bou}. In addition,  $L(\l)$ will denote the irreducible $K\bar{S}$-module with highest weight $\l$. Note that if $M$ is an irreducible $K\hat{S}$-module and $\gamma$ is an automorphism of $\hat{S}$, then the highest weight of $M^{\gamma}$ can be read off from \cite[Proposition 5.4.2]{KL}. Similarly, the highest weight of the dual module $M^*$ is described in \cite[Proposition 5.4.3]{KL}. We refer the reader to \cite[Section 5.4]{KL} and the references therein for further details.

\subsubsection{$S$ is untwisted}\label{ss:untw}

To begin with, we will assume $S$ is an untwisted simple group of Lie type over $\mathbb{F}_{p^e}$. 
Recall that $T$ is a finite simple classical group over $\mathbb{F}_{q}$ with natural module $V$, where $q=p^f$. Set $q'=p^{f'}$, where $f'=2f$ if $T={\rm PSU}_{n}(q)$, otherwise $f'=f$. Also recall that $V$ is an absolutely irreducible $\mathbb{F}_{q'}\hat{S}$-module which cannot 
be realised over a proper subfield of $\mathbb{F}_{q'}$ (see Definition \ref{sdef}). By applying \cite[Proposition 5.4.6(i)]{KL} we deduce that $f'$ divides $e$ and there exists an irreducible $K\hat{S}$-module $M$ such that 
\begin{equation}\label{e:tensor}
\bar{V} = V \otimes K \cong M \otimes M^{(f')} \otimes M^{(2f')} \otimes \cdots \otimes M^{(e-f')}
\end{equation}
(with $e/f'$ factors) as $K\hat{S}$-modules. 
Set $\ell = \dim M$ and note that $\ell \geqs 2$ and $n=\ell^{e/f'}$. 

We need a couple of preliminary lemmas.

\begin{lem}\label{l:tensor}
Let $J$ be a finite group and let $V_1$ and $V_2$ be faithful finite dimensional $KJ$-modules, where $K$ is an algebraically closed field and $\dim V_i \geqs 2$, $i=1,2$. Let $x \in J$ be a nontrivial element such that the action of $x$ on $V_1$ has a repeated eigenvalue. Then $x$ has a nontrivial repeated eigenvalue on $V_1 \otimes V_2$.
\end{lem}

\begin{proof}
This is an easy exercise.
\end{proof}

\begin{lem}\label{l:sldd}
Let $\bar{S} = {\rm SL}_{d}(K)$, where $d \geqs 6$ and $K$ is an algebraically closed field of characteristic $p \geqs 0$. Let $\bar{V} = L(\l)$ be an $n$-dimensional irreducible self-dual $K\bar{S}$-module with $n \leqs 4d^2$. Then either $\bar{V}$ is the adjoint module, or 
\begin{equation}\label{e:cas}
(d,\l) \in \{ (6,\l_3),  (6,2\l_3), (8,\l_4), (10,\l_5) \}
\end{equation}
up to quasiequivalence.
\end{lem}

\begin{proof}
We follow the proof of \cite[Proposition 2.5]{Bur}. Write $\l=\sum_{j=1}^{d-1}a_j\l_j$ where each $a_j$ is a non-negative integer. By self-duality, we have $a_{j} = a_{d-j}$ for all $j$. To begin with, let us assume that $\l$ is $p$-restricted (that is, $a_j<p$ for all $j$). If $d \leqs 18$ then the result can be checked by inspecting the relevant tables in \cite{Lu}, so we may assume that $d \geqs 19$. Let $\mathcal{W} \cong S_d$ be the Weyl group of $\bar{S}$, which acts naturally on the set of weights of $\bar{V}$.

Suppose $a_2 \ne 0$. By arguing as in the proof of \cite[Proposition 2.5]{Bur} we see that  the $\mathcal{W}$-stabiliser of $\l$ is contained in a parabolic subgroup of type $A_1 \times A_{d-5} \times A_1$ and thus
$$n \geqs |\mathcal{W}\cdot \l| = |\mathcal{W}:\mathcal{W}_{\lambda}| \geqs \frac{|S_d|}{|S_2|^2|S_{d-4}|} = \frac{1}{4}d(d-1)(d-2)(d-3)> 4d^2,$$
where $\mathcal{W}\cdot \l$ denotes the $\mathcal{W}$-orbit of $\l$. Therefore $a_2=a_{d-2}=0$. In this way, we quickly reduce the problem to the case where
$$\l = \left\{\begin{array}{ll}
a\l_1 + a\l_{d-1} & \mbox{$d$ odd} \\
a\l_1 + b\l_{d/2}+a\l_{d-1} & \mbox{$d$ even}
\end{array}\right.$$
If $d$ is even and $b \ne 0$ then the $\mathcal{W}$-stabiliser of $\l$ is contained in a parabolic subgroup of type $A_{d/2-1} \times A_{d/2-1}$ and thus $n \geqs d!/((d/2)!^2)>4d^2$. Finally, we can repeat the argument in the proof of \cite[Proposition 2.5]{Bur} to see that $a=1$ is the only option, so $\l=\l_1+\l_{d-1}$ and thus $\bar{V}$ is the adjoint module.

Finally, let us relax the assumption that $\l$ is $p$-restricted. Write $\l = \mu_0+p\mu_1+\cdots+p^{e-1}\mu_{e-1}$, where each $\mu_i$ is $p$-restricted, so by Steinberg's tensor product theorem we have
$$\bar{V} = L(\l) \cong L(\mu_0) \otimes L(\mu_1)^{(1)} \otimes \cdots \otimes L(\mu_{e-1})^{(e-1)}.$$
If three or more of the $\mu_i$ are nonzero then $n \geqs d^3>4d^2$, which is a contradiction. Next suppose two are nonzero, say $\l = p^i\mu_i+p^j\mu_j$ with $i \ne j$, so $n = \dim L(\mu_i) \cdot \dim L(\mu_j)$. The self-duality of $\bar{V}$ implies that $L(\mu_i)$ and $L(\mu_j)$ are also self-dual and thus the result in the $p$-restricted case rules out this situation for dimension reasons. Finally, if $\l = p^i\mu_i$ then $\mu_i$ is self-dual and $\bar{V}$ is quasiequivalent to $L(\mu_i)$. The result follows.
\end{proof}

\begin{lem}\label{p:untwist0}
Proposition \ref{p:c} holds if $S$ is untwisted and $e>f'$.
\end{lem}

\begin{proof}
First assume $\ell>2$, where $\ell$ denotes the dimension of $M$ in \eqref{e:tensor}. Fix an element $x \in T$ of order $r$ with $\nu(x) = c$, where $r \ne p$ and $r>2$ (see Remark \ref{r:cc}). We claim that $x$ is a derangement. In order to see this,  
we need to show that if $g \in H_0$ has order $r$, then $g$ is not $T$-conjugate to $x$. For instance, it suffices to show that $\nu(g) \ne c$, or that $g$ has a nontrivial repeated eigenvalue on $\bar{V}$. 

Let $g \in H_0$ be an element of order $r$. If $g$ is a field automorphism then it must induce a fixed point free permutation on the $e/f'$ factors in the tensor product decomposition \eqref{e:tensor} (in particular, $r$ divides $e/f'$). This implies that $g$ has nontrivial repeated eigenvalues on $\bar{V}$, so it is not conjugate to $x$. To complete the argument, we may assume that $g$ is an inner-diagonal automorphism (recall that $r \geqs 5$) and thus $g$ stabilises each of the tensor factors in \eqref{e:tensor}. Let $\nu_1(g)$ and $\nu(g)$ denote the codimension of the largest eigenspace of $g$ on $M$ and $\bar{V}$, respectively.
By applying \cite[Lemma 3.7]{LieShalev}, we deduce that
\begin{equation}\label{e:nuV}
\nu(g) \geqs \nu_1(g)n/\ell.
\end{equation}
If $\nu_1(g)<\ell-1$ then Lemma \ref{l:tensor} implies that $g$ has a nontrivial repeated eigenvalue on $\bar{V}$, so we may assume that $\nu_1(g)=\ell-1$.
Then \eqref{e:nuV} gives
$\nu(g)>n/2$
and thus $g$ is not $T$-conjugate to $x$ (since $\nu(x) = c \leqs n/2$).  

Now assume $\ell=2$, so $S = {\rm PSL}_{2}(p^e)$ is the only possibility. The previous  argument shows that the element $x \in T$ above is a derangement if $c<n/2$, so we may assume that $c=n/2$. Here \eqref{r:bd} implies that 
$$r \geqs c+1 = n/2+1 = 2^{e/f'-1}+1>e/f',$$
so every element in $H_0$ of order $r$ is contained in $S$ (indeed, if $g \in H_0 \setminus S$ has order $r$, then $g$ is a field automorphism and $r$ divides $e/f'$, as noted above). Since $S$ has a unique conjugacy class of subgroups of order $r$, we conclude that $T$ contains derangements of order $r$.
\end{proof}

\begin{lem}\label{p:untwist}
Proposition \ref{p:c} holds if $S$ is untwisted and $e=f'$. 
\end{lem}

\begin{proof}
Set $q'=p^{f'}$ as before, so $f'=2f$ if $T={\rm PSU}_{n}(q)$, otherwise $f'=f$. Here 
$\bar{V} \cong M$ for some irreducible $K\hat{S}$-module $M$, which is not quasiequivalent to the natural module for $\hat{S}$ (see Definition \ref{sdef}). Note that every element of order $r$ in $H_0$ is inner-diagonal.

First assume $T = {\rm PSU}_{n}(q)$, in which case $M \cong (M^{*})^{(f)}$, where $M^*$ denotes the dual of $M$ (see \cite[Lemma 2.10.15(ii)]{KL}, for example). By considering this isomorphism at the level of highest weights,  
and by applying Steinberg's tensor product theorem, we deduce that $M$ is isomorphic to a tensor product of two or more nontrivial irreducible $K\hat{S}$-modules. For example, if $S = {\rm PSL}_{3}(q^2)$ and $M$ has highest weight $\l_1+q\l_{2}$, then $M \cong L(\l_1) \otimes L(\l_2)^{(f)}$ is $9$-dimensional and $M \cong (M^{*})^{(f)}$, so this yields an embedding of $S$ in ${\rm PSU}_{9}(q)$. By expressing $M$ as a tensor product in this way, we can repeat the argument in the proof of Lemma \ref{p:untwist0} to see that every element $x \in T$ of order $r$ with $\nu(x)=c$ is a derangement. 

For the remainder of the proof we may assume that $T \ne {\rm PSU}_{n}(q)$, so $q=q'$. We will start by assuming $S$ is a classical group, with a $d$-dimensional natural module. Set $i = \Phi(r,q)$ and 
\begin{equation}\label{e:eqcdash}
c' = \left\{\begin{array}{ll}
2i & \mbox{if $i$ is odd and $S \neq {\rm PSL}_d(q)$} \\
i/2 &  \mbox{if $i\equiv 2\imod{4}$ and $S = {\rm PSU}_d(q)$} \\
i & \mbox{otherwise}
\end{array}\right.
\end{equation}
Note that if $c'>d/2$ then $\kappa(S,r)=1$ by Lemma \ref{l:floor}(i), so we may assume that $c' \leqs d/2$. In addition, note that if $c' \geqs c$ and $n>2d$ then Lemma \ref{l:csub} implies that $\kappa(T,r)>\kappa(S,r)$ and thus $T$ contains derangements of order $r$.

\vs

\noindent \emph{Case 1. $S={\rm PSp}_{d}(q)$, $d \geqs 4$}

\vs

First observe that $V$ is self-dual and thus $T$ is a symplectic or orthogonal group (see \cite[Lemma 2.10.15(i) and Proposition 5.4.3]{KL}). In particular, $c'=c$ is even, so $c' \geqs 4$ and thus $d \geqs 8$ since we are assuming that $c' \leqs d/2$. If $d \geqs 12$ then \cite[Theorems 4.4 and 5.1]{Lu} imply that $n \geqs (d^2-d-4)/2>2d$ and the result follows from Lemma \ref{l:csub}. Similarly, if $d \in \{8,10\}$ then $n \geqs 2^{d/2}$ (see \cite[Tables A.33 and A.34]{Lu}) and we reduce to the case $d=8$ with $n=16$. Here $c'=c=4$ and it is easy to see that $\kappa(S,r)<\kappa(T,r)$.

\vs

\noindent \emph{Case 2. $S \in \{{\rm P\Omega}_{d}^{+}(q),\Omega_d(q)\}$, $d \geqs 7$}

\vs

First assume that $V$ is self-dual, so $T$ is symplectic or orthogonal, and $c'=c$ is even. In particular, note that $d \geqs 8$. Suppose $d$ is even. If $n \geqs (d^2-d-4)/2$ then $n>2d$ and the result follows as in Case 1. Therefore, by applying \cite[Theorems 4.4 and 5.1]{Lu}, we may assume that $d \in \{8,12\}$ and $n=2^{d/2-1}$. If $d=12$ then $n>2d$ and the result follows as before. We can discard the case $d=8$ since $S \not\cong T$. Now assume $d \geqs 9$ is odd. By arguing as above we may assume that $n < (d^2-d-4)/2$, so $d \leqs  23$ by \cite[Theorem 5.1]{Lu}. In each of the remaining cases it is easy to check that $n>2d$ by inspecting the relevant tables in \cite[Appendix A]{Lu}, unless $d=9$ and $n=16$. In the latter case, $c'=c=4$ and $\kappa(S,r)<\kappa(T,r)$. Finally, if $V$ is not self-dual then $S = {\rm P\Omega}_{d}^{+}(q)$ with $d \equiv 2 \imod{4}$ (see \cite[Proposition 5.4.3]{KL}) and $T = {\rm PSL}_{n}(q)$. Here $c' \geqs c$ and the above argument goes through. 

\vs

\noindent \emph{Case 3. $S={\rm PSL}_{d}(q)$, $d \geqs 2$}

\vs

If $d=2$ then $\kappa(S,r)=1$ so we may assume that $d \geqs 3$. If $V$ is not self-dual, then $T = {\rm PSL}_{n}(q)$, $c'=c=i$ and $d \geqs 6$. By applying \cite[Theorems 4.4 and 5.1]{Lu} we see that $n \geqs d(d-1)/2>2d$. Similarly, if $V$ is self-dual and $i$ is even, then $c'=c=i \leqs d/2$ and $n \geqs d^2-2>2d$. In both cases, the desired result follows from Lemma \ref{l:csub}. 

Finally, let us assume $V$ is self-dual and $i$ is odd. Here $c'=i<2i=c$ so we cannot appeal to Lemma \ref{l:csub}. First observe that $i \geqs 3$ and thus $d \geqs 6$. Also recall that $c \geqs \sqrt{n}/2$ and $c' \leqs d/2$, hence $n \leqs 4d^2$ and the possibilities for $V$ are recorded in Lemma \ref{l:sldd}.

First let us consider the exceptional cases in \eqref{e:cas}. Suppose $(d,\bar{V}) = (6,L(\l_3))$. Here $i=3$ and $V = \L^3(W)$ is $20$-dimensional, where $W$ is the natural $S$-module. A straightforward calculation shows that $\nu(y) \geqs 8$ for all $y \in S$ of order $r$ (see \cite[Section 7]{Burr}, for example), so every element $x \in T$ of order $r$ with $\nu(x)=6$ is a derangement. A very similar argument applies if $(d,\bar{V}) = (8,L(\l_4))$ or $(10,L(\l_5))$. Finally, suppose $(d,\bar{V}) = (6,L(2\l_3))$, so $p=3$ and $n=141$ (see \cite[Table A.9]{Lu}). Again, $i=3$ and thus $6$ divides $r-1$. Set $a=(r-1)/6$ and observe that $S$ has $4a+\binom{2a}{2}$ conjugacy classes of elements of order $r$. If $r=7$ then $T$ has $\lfloor 141/6 \rfloor = 23>5$ such classes, so we may assume that $r \geqs 13$. It is easy to check that $T$ has at least $2a+22\binom{a}{2}$ such classes, and the result follows by applying Corollary \ref{c:count}. (To obtain the latter lower bound, we simply count class representatives of the form $[X_1,I_{135}]$, $[X_1^2,I_{129}]$ and $[X_1^j,X_2,I_{141-6(j+1)}]$ with $1 \leqs j \leqs 22$.)

Now assume $V$ is the adjoint module, so $n = d^2-1$ or $d^2-2$ (according to whether or not $p$ divides $d$). Let $X$ be the Lie algebra of $\bar{S} = {\rm SL}_{d}(K)$, so $\bar{V}$ is the nontrivial irreducible constituent of $X$. Now $\dim C_X(y) = \dim C_{\bar{S}}(y)$ for every nontrivial semisimple element $y \in \bar{S}$ (see \cite[Section 1.10]{HCC}) and thus \cite[Proposition 2.9]{Burr} implies that 
$$\dim C_X(y) = \dim C_{\bar{S}}(y) \leqs d^2-2d+1.$$ 
It follows that the dimension of the $1$-eigenspace of any element in $S$ of order $r$ on $V$ is at most $d^2-2d+1$. 
But if $x \in T$ is an element of order $r$ with $\nu(x)=c$, then 
$$\dim C_V(x) = n-c \geqs d^2-d-2>d^2-2d+1$$
and we conclude that $x$ is a derangement.

\vs

\noindent \emph{Case 4. $S=E_8(q)$}

\vs

As before, $H_0$ does not contain any field automorphisms, so by considering the order of $S$ we deduce that $c \leqs 30$ and thus $n \leqs 3600$ since we have $c \geqs \sqrt{n}/2$. By inspecting \cite[Table A.53]{Lu}, we deduce that $n=248$ is the only possibility, so $V$ is the adjoint module. In particular, since $\bar{V}$ is the Lie algebra of $\bar{S}=E_8(K)$ we have  
$$\dim C_{\bar{V}}(y) = \dim C_{\bar{S}}(y) \leqs 136$$
for all nontrivial semisimple elements $y \in \bar{S}$. Therefore, every $x \in T$ of order $r$ with $\nu(x)=c$ is a derangement.

\vs

\noindent \emph{Case 5. $S=E_7(q)$}

\vs

Here $c \leqs 18$ and thus $n \leqs 1296$. Suppose $c \in \{14,18\}$. By inspecting the structure of the maximal tori in $S$ (see \cite[Section 2.9]{KS}, for example), we deduce that every element in $S$ of order $r$ belongs to a unique conjugacy class of maximal tori, which are cyclic. Since such a torus has a unique subgroup of order $r$, it follows that $\kappa(S,r)=1$ and thus $T$ contains derangements of order $r$.

By inspecting the order of $S$, we may assume that $c \leqs 12$ and thus $n \leqs 576$. By \cite[Table A.52]{Lu}, it follows that $n \in \{56,132,133\}$, so $V$ is either the minimal or adjoint module for $S$. Suppose 
$V$ is the adjoint module and let $X$ be the Lie algebra of $\bar{S}=E_7(K)$. Then 
$$\dim C_X(y) = \dim C_{\bar{S}}(y) \leqs 79$$
for all nontrivial semisimple elements $y \in \bar{S}$, so every $x \in T$ of order $r$ with $\nu(x)=c$ is a derangement. Finally, suppose $n=56$ and consider the restriction of $\bar{V}$ to a maximal rank subgroup $A_7$ of $\bar{S}$. By \cite[Proposition 2.3]{LSM} we have
$$\bar{V} \downarrow A_7 = L(\l_2) \oplus L(\l_2)^* = \Lambda^2(W) \oplus \Lambda^2(W)^*,$$
where $W$ is the natural $A_7$-module. By calculating directly with the exterior square $\L^2(W)$ we find that $\dim C_{L(\l_2)}(y) \leqs 21$ for all nontrivial semisimple elements $y \in A_7$, so the $1$-eigenspace of any element in $S$ of order $r$ on $V$ has dimension at most $42$. Since $c \leqs 12$, we conclude that each $x \in T$ of order $r$ with $\nu(x)=c$ is a derangement.

\vs

\noindent \emph{Case 6. $S=E_6(q)$}

\vs

This is very similar to the previous case. First observe that $c \leqs 12$ and thus $n \leqs 576$. If $c\in \{9,12\}$ then by considering the maximal tori of $S$ we deduce that $\kappa(S,r)=1$ and the result follows. In view of $|S|$, we may assume that $c \leqs 8$, so $n \leqs 256$ and thus $n \in \{27,77,78\}$ by \cite[Table A.51]{Lu}. If $n \in \{77,78\}$ then $V$ is the adjoint module and we see that every $x \in T$ of order $r$ with $\nu(x)=c$ is a derangement since $\dim C_X(y) = \dim C_{\bar{S}}(y) \leqs 46$ for all nontrivial semisimple elements $y \in \bar{S}$ (where $X$ is the Lie algebra of $\bar{S}=E_6(K)$). 

Finally, let us assume $c \leqs 8$ and $n=27$, so $V$ is the minimal module for $S$. Once again we claim that every element $x \in T$ of order $r$ with $\nu(x)=c$ is a derangement. To see this, first observe that  
$$\bar{V} \downarrow A_1A_5 = \left(L(\l_1) \otimes L(\l_1)\right) \oplus \left(0 \otimes L(\l_4) \right) = U_1 \oplus U_2,$$
where $0$ denotes the trivial $A_1$-module (see \cite[Proposition 2.3]{LSM}). Let $y=y_1y_2 \in A_1A_5$ be a nontrivial semisimple element. If one $y_j$ is trivial then it is clear that $y$ has a repeated nontrivial eigenvalue on $\bar{V}$. On the other hand, if both $y_1$ and $y_2$ are nontrivial then we calculate that $\dim C_{U_1}(y) \leqs 6$ and $\dim C_{U_2}(y) \leqs 10$ (note that $U_2 \cong \Lambda^2(W)^*$, where $W$ is the natural $A_5$-module), so $\dim C_{\bar{V}}(y) \leqs 16$. This justifies the claim. 

\vs

\noindent \emph{Case 7. $S=F_4(q)$}

\vs

Here $c \leqs 12$ and thus $n \leqs 576$. If $c \in \{8,12\}$ then by considering the structure of the maximal tori of $S$ we see that $\kappa(S,r)=1$, so we may assume that $c \leqs 6$. In particular, $n \leqs 144$ and thus $n \in \{25,26,52\}$ (see \cite[Table A.50]{Lu}). We claim that each $x \in T$ of order $r$ with $\nu(x)=c$ is a derangement.

If $n=52$ then $\bar{V}$ is the Lie algebra of $\bar{S}=F_4(K)$, so $\dim C_{\bar{V}}(y) \leqs 36$ and the claim follows. Finally, if $n \in \{25,26\}$ then the proof of \cite[Lemma 7.4]{Burr} implies that $\nu(y) \geqs 7$ for all nontrivial semisimple elements $y \in \bar{S}$, and once again the claim holds.

\vs

\noindent \emph{Case 8. $S=G_2(q)$}

\vs

Since $c \geqs 3$ and $|S|=q^6(q^2-1)(q^6-1)$, we see that $c = 6$ is the only possibility. By inspecting the maximal tori of $S$, we deduce that $\kappa(S,r)=1$  and the result follows.
\end{proof}

\subsubsection{$S$ is twisted}\label{ss:tw1}

To complete the proof of Proposition \ref{p:c} (and hence the proof of Theorem \ref{t:main}), we may assume that $S$ is a twisted group of Lie type over $\mathbb{F}_{p^e}$. 

For now, let us assume that $S$ is of type ${\rm PSU}_{d}(p^e)$ (with $d \geqs 3$), ${\rm P\Omega}_{d}^{-}(p^e)$ (with $d \geqs 8$) or ${}^2E_6(p^e)$. In each of these cases, the ambient simple algebraic group admits a graph automorphism $\tau$ of order $2$, which induces a symmetry of the corresponding Dynkin diagram. We will also write $\tau$ to  denote the restriction of this automorphism to the corresponding twisted group $\hat{S}$. Recall that if $M$ is a $K\hat{S}$-module affording the representation $\rho$, then $M^{\tau}$ denotes the space $M$ with $\hat{S}$ acting via $\tau\rho$.

As before, set $q=p^f$ and $q'=p^{f'}$, where $f'=2f$ if $T={\rm PSU}_{n}(q)$, otherwise $f'=f$. Since $V$ is an absolutely irreducible $\mathbb{F}_{q'}\hat{S}$-module which cannot 
be realised over a proper subfield of $\mathbb{F}_{q'}$, \cite[Proposition 5.4.6(ii)]{KL} implies that one of the following occurs:
\begin{itemize}\addtolength{\itemsep}{0.2\baselineskip}
\item[(a)] $f'$ divides $e$ and there is an irreducible $K\hat{S}$-module $M$ such that $M^{\tau} \cong M$ and \eqref{e:tensor} holds.
\item[(b)] $f'$ divides $2e$, but $f'$ does not divide $e$. Moreover, if we write $\bar{V}  = V \otimes K$ then there is an irreducible $K\hat{S}$-module $M$ such that $M^{\tau} \not\cong M$ and
\begin{equation}\label{e:tensor2}
\bar{V} 
 \cong M \otimes (M^{\tau})^{(f'/2)} \otimes M^{(f')} \otimes (M^{\tau})^{(3f'/2)} \otimes \cdots \otimes (M^{\tau})^{(e-f')} \otimes M^{(e-f'/2)}
\end{equation}
(with $2e/f'$ factors) as $K\hat{S}$-modules. 
\end{itemize}
Set $\ell = \dim M$ and note that $\ell \geqs 3$. Also note that $n=\ell^{e/f'}$ in (a), and $n = \ell^{2e/f'}$ in (b).

\begin{lem}\label{p:twist}
Proposition \ref{p:c} holds if $S$ is of type ${\rm PSU}_{d}(p^e)$, ${\rm P\Omega}_{d}^{-}(p^e)$ or ${}^2E_6(p^e)$.
\end{lem}

\begin{proof}
First let us assume that we are in case (a) above, so $f'$ divides $e$ and \eqref{e:tensor} holds with respect to an irreducible $K\hat{S}$-module $M$ such that $M^{\tau} \cong M$. If $f'<e$ then the proof of Lemma \ref{p:untwist0} goes through unchanged (note that we always have $\ell >2$) and we deduce that every element $x \in T$ of order $r$ with $\nu(x)=c$ is a derangement. 

Now assume that (a) holds and $f'=e$. Here $\bar{V} \cong M \cong M^{\tau}$, so $V$ is self-dual (see \cite[Proposition 5.4.3]{KL}). In particular, $T$ is either symplectic or orthogonal, and $q=q'$. Also note that every element in $H_0$ of order $r$ is inner-diagonal. Set $i=\Phi(r,q)$ as before and note that $c=2i$ if $i$ is odd, otherwise $c=i$. Define the integer $c'$ as in \eqref{e:eqcdash}. As noted in the proof of Lemma \ref{p:untwist}, we may assume that $c' \leqs d/2$. In addition, Lemma \ref{l:csub} implies that if $c' \geqs c$ then it suffices to show that $n>2d$.

Suppose $S = {\rm PSU}_{d}(q)$ with $d \geqs 3$. First assume that $i \not\equiv 2 \imod{4}$, so $3 \leqs c=c' \leqs d/2$ and \cite{Lu} implies that $n \geqs d(d-1)/2>2d$ as required. Now assume $i \equiv 2 \imod{4}$, so $c=i$ and $c'=i/2$. Since $c \geqs 3$ and $c' \leqs d/2$ we have $i,d \geqs 6$. In addition, since $c \geqs \sqrt{n}/2$ we deduce that $n \leqs 4d^2$. The rest of the argument is now identical to the analysis in Case 3 in the proof of Lemma \ref{p:untwist}. The reader can check the details.

Next assume $S = {\rm P\Omega}_{d}^{-}(q)$ and $d \geqs 8$. Here $c'=c$ and we can repeat the argument in Case 2 in the proof of Lemma \ref{p:untwist}. Finally, let us assume that $S = {}^2E_6(q)$. By inspecting the order of $S$ we see that $c \leqs 18$. If $c>8$ then by considering the structure of the maximal tori of $S$ we deduce that $\kappa(S,r)=1$ and the result follows. Now assume $c \leqs 8$ so $n \leqs 256$. By inspecting \cite[Table A.51]{Lu} we see that $n \in \{27,77,78\}$ and we can now repeat the argument presented in Case 6 in the proof of Lemma \ref{p:untwist}.

To complete the proof of the lemma we may assume that (b) holds so $f'$ divides $2e$, but $f'$ does not divide $e$, and there is an irreducible $K\hat{S}$-module $M$ such that $M^{\tau} \not\cong M$ and \eqref{e:tensor2} holds. Note that $n = \ell^{2e/f'}$, where $\ell= \dim M$. If $f'<2e$ then the argument in the proof of Lemma \ref{p:untwist0} goes through, so we may assume that $f'=2e$ and thus $\bar{V} \cong M$. Also note that $V^{(e)} \cong V^{\tau}$ (see \cite[Proposition 5.4.6(ii)]{KL}).

First assume $S = {\rm PSU}_{d}(p^e)$ and $d \geqs 3$. Here $V^{(e)} \cong V^{\tau} \cong V^*$ and thus $(V^*)^{(e)} \cong V$, so $T={\rm PSU}_{n}(q)$, $f'=2f$ and $S = {\rm PSU}_{d}(q)$. In particular, $c'=c$, $d \geqs 6$ and $n \geqs d(d-1)/2>2d$, so the result follows from Lemma \ref{l:csub}.

Next assume $S = {\rm P\Omega}_{d}^{-}(p^e)$ with $d \geqs 8$. Suppose $d \equiv 0 \imod{4}$. Then $V$ is self-dual (see \cite[Proposition 5.4.3]{KL}), so $T$ is a symplectic or orthogonal group and thus $S = {\rm P\Omega}_{d}^{-}(q_0)$ with $q=q_0^2$. Note that $n>d$ (if $n=d$ then $V$ is the natural module for $\hat{S}$, which is defined over a proper subfield of $\mathbb{F}_{q}$). Set $i = \Phi(r,q)$ and $i_0 = \Phi(r,q_0)$, so $i = i_0/2$ if $i_0$ is even, otherwise $i=i_0$. Also set $c'=2i_0$ if $i_0$ is odd, otherwise $c'=i_0$. Then $c'=2i \geqs c$ and by arguing as in Case 2 in the proof of Lemma \ref{p:untwist} we deduce that $n>2d$.

Now suppose $S = {\rm P\Omega}_{d}^{-}(p^e)$, where $d \geqs 10$ and $d \equiv 2 \imod{4}$. In this situation, $V$ is not self-dual. In fact, $(V^*)^{(e)} \cong V$ and thus $T={\rm PSU}_{n}(q)$ and $S = {\rm P\Omega}_{d}^{-}(q)$. Therefore $c' \geqs c$ and once again it is easy to check that $n>2d$.

Finally, let us assume that $S = {}^2E_6(p^e)$. As in the previous case, we have $T = {\rm PSU}_{n}(q)$ and $S = {}^2E_6(q)$. Note that every element of order $r$ in $H_0$ is contained in $S$. 
If $c>8$ then $c \in \{9,12\}$, $i \in \{12,18\}$ and thus $\kappa(S,r)=1$. On the other hand, if $c \leqs 8$ then $n \leqs 256$ and we can proceed as in Case 6 in the proof of Lemma \ref{p:untwist}.
\end{proof}

We now complete the proof of Proposition \ref{p:c} by dealing with the remaining twisted groups.

\begin{lem}\label{p:twist2}
Proposition \ref{p:c} holds if $S$ is of type ${}^3D_{4}(p^e)$, ${}^2B_{2}(2^e)$, ${}^2G_{2}(3^e)$ or ${}^2F_{4}(2^e)$.
\end{lem}

\begin{proof}
Set $q=p^f$ and note that $V$ is self-dual (see \cite[p.192]{KL}), so $T = {\rm PSp}_{n}(q)$ or ${\rm P\Omega}_{n}^{\e}(q)$. As usual, we set $H_0 = H \cap T$ and $i=\Phi(r,q)$. We partition the proof into several cases.

\vs

\noindent \emph{Case 1. $S = {}^3D_{4}(p^e)$}

\vs

Set $t=p^e$ and note that $|S|=t^{12}(t^8+t^4+1)(t^6-1)(t^2-1)$. Since $r \geqs 5$, every element in $H$ of order $r$ is contained in $S.\la \varphi \ra$, where $\varphi$ is a field automorphism of order $r$. There are $r-1$ distinct $S$-classes of field automorphisms of order $r$ in ${\rm Aut}(S)$, represented by the elements 
$\varphi^j$ with $1 \leqs j < r$ (this follows from the fact that every element of order $r$ in the coset $S\varphi^j$ is $S$-conjugate to $\varphi^j$ -- see \cite[Proposition 4.9.1(d)]{GLS}). Therefore, there is at most one $S$-class of subgroups of order $r$ with elements in $H_0\setminus S$, so we may assume that $r$ divides $|S|$. As noted in \cite[Remark 5.4.7(a)]{KL}, either $f$ divides $e$, or $f$ divides $3e$ (and $f$ does not divide $e$).

First assume $f$ divides $e$, so $t=q^{e/f}$. According to \cite[Remark 5.4.7(a)]{KL}, there exists an irreducible $K\hat{S}$-module $M$ such that $M^{\tau} \cong M$ and
$$\bar{V} = V \otimes K \cong M \otimes M^{(f)} \otimes M^{(2f)} \otimes \cdots \otimes M^{(e-f)},$$
where $\tau$ denotes a triality graph automorphism of $\hat{S}$ of order $3$. Note that the condition $M^{\tau} \cong M$ implies that $\dim M \geqs 26$ (see \cite[Table A.41]{Lu}, for example), so $n \geqs 26^{e/f}$.

Suppose $r$ divides $t^8+t^4+1$.  Since $r$ divides $q^{12e/f}-1$ it follows that $i$ divides $12e/f$. Therefore,
$$12e/f \geqs c \geqs \lceil \sqrt{n}/2 \rceil \geqs \lceil \sqrt{26^{e/f}}/2 \rceil$$
and thus $e/f \leqs 2$. In particular, $e/f$ is indivisible by $r$, so $H_0$ does not contain any field automorphisms of order $r$. By inspecting the structure of the maximal tori of $S$ (see \cite[Section 2.4]{KS}) we deduce that $\kappa(S,r)=1$ and the result follows. 

Next assume $r$ divides $t^6-1$. Here $c \leqs 6e/f$ and by arguing as above we deduce that $e/f \leqs 2$. If $e/f = 2$ then $12 \geqs c \geqs \lceil 26/2 \rceil=13$, which is absurd, so we may assume that $e/f=1$, hence $i \in \{3,6\}$ and $c=6$. Moreover, the bound $6 \geqs \lceil \sqrt{n}/2 \rceil$ implies that $n \leqs 144$. By inspecting \cite[Table A.41]{Lu}, using the fact that the highest weight of $M$ is fixed under the induced action of $\tau$ on weights, it follows that $n = 28 - 2\delta_{2,p}$ and $V$ is the adjoint module. Let $X$ be the Lie algebra of $\bar{S} = D_4$ and observe that 
$$\dim C_X(y) = \dim C_{\bar{S}}(y) \leqs 16$$
for all nontrivial semisimple elements $y \in \bar{S}$ (indeed, the proof of \cite[Proposition 2.9]{Burr} implies that $\dim y^{\bar{S}} \geqs 12$). We immediately deduce that every element $x \in T$ of order $r$ with $\nu(x)=c$ is a derangement.

To complete the analysis of the case $S={}^3D_4(p^e)$ we may assume that $f$ divides $3e$, but $f$ does not divide $e$. Here there is an irreducible $K\hat{S}$-module $M$ such that $M^{\tau} \not\cong M$ and
$$V \otimes K \cong M \otimes (M^{\tau})^{(f/3)} \otimes (M^{\tau^2})^{(2f/3)} \otimes M^{(f)} \otimes \cdots$$
(with $3e/f$ factors in total). Note that $\dim M \geqs 8$.

Suppose $r$ divides $t^8+t^4+1$, in which case $r$ divides $q^{12e/f}-1$ and thus 
$$12e/f \geqs c \geqs \lceil \sqrt{n}/2 \rceil \geqs \lceil \sqrt{8^{3e/f}}/2 \rceil$$
since $n \geqs 8$. This implies that $3e/f=1$ or $2$. In particular, $r$ does not divide $3e/f$, so $\kappa(H_0,r)=1$ and the result follows. Finally, let us assume that $r$ divides $t^6-1$, so $c \leqs 6e/f$ and we deduce that $3e/f \leqs 2$ since $n \geqs 8$. If $3e/f=1$ then $c=2$, which is a contradiction. If $3e/f=2$ then $c=4$ and $n=8$ is the only possibility, but this can be ruled out by inspecting the relevant tables in \cite[Section 8.2]{BHR}.

\vs

\noindent \emph{Case 2. $S = {}^2B_{2}(2^e)$}

\vs

Set $t=2^e$, where $e \geqs 3$ is odd, and note that $|S|=t^2(t^2+1)(t-1)$ and ${\rm Aut}(S) = S.\la \phi \ra$, where $\phi$ is a field automorphism of order $e$. 
Now $S$ has exactly three conjugacy classes of maximal tori, all of which are cyclic. By considering the orders of the maximal tori, we deduce that $\kappa(S,r)=1$ for every odd prime divisor $r$ of $|S|$. As in the previous case, there is at most one $S$-class of subgroups of order $r$ with elements in $H_0 \setminus S$, whence $\kappa(H_0,r) \leqs 2$. The desired result follows immediately if $\kappa(H_0,r)=1$, so we may assume that $r$ divides $|S|$.  

Write $q=2^f$ and note that $f$ divides $e$ and $n = \dim V \geqs 4^{e/f}$ (see \cite[Remark 5.4.7(b)]{KL}). Since $r$ divides $|S|$, it divides either $t-1$ or $t^2+1$. Suppose $r$ divides $t-1 = q^{e/f}-1$, so $i$ divides $e/f$ and thus $i$ is odd, so
$$2e/f \geqs 2i = c \geqs \lceil \sqrt{n}/2 \rceil \geqs 2^{e/f-1}$$
and it follows that $e/f = 1$ or $3$. But we are assuming that $c \geqs 3$, so $e/f=3$ and thus $i=3$ and $c=6$. Since $n \geqs 4^{e/f} = 64$, we deduce that $\kappa(T,r) \geqs 3$ and thus $T$ contains derangements of order $r$. A similar argument applies when $r$ divides $t^2+1$. Here $r$ divides $q^{2e/f}+1$ and thus $i$ divides $4e/f$. If $i$ is odd then $i$ divides $e/f$, so $r$ divides $q^{e/f}-1$, which is not possible. Therefore, $i$ is even and thus
$$4e/f \geqs i = c \geqs \lceil \sqrt{n}/2 \rceil \geqs 2^{e/f-1},$$
so $e/f \in \{1,3,5\}$. If $e/f=3$ or $5$ then the bound $n \geqs 4^{e/f}$ quickly implies that $\kappa(T,r) \geqs 3$. Finally, if $e/f=1$ then $c=4$ and $n \geqs 16$ (indeed, every absolutely irreducible representation of $S$ over a field of characteristic $2$ has dimension $4^m$ for some $m$; see \cite[Lemma 1]{Mart}, for example) and once again we conclude that $\kappa(T,r) \geqs 3$.

\vs

\noindent \emph{Case 3. $S = {}^2G_2(3^e)$}

\vs

Write $t=3^e$, where $e \geqs 3$ is odd, and note that $|S| = t^3(t^3+1)(t-1)$. By inspecting the 
structure of the maximal tori of $S$, we deduce that $\kappa(S,r)=1$ for every prime $r \geqs 5$ dividing $|S|$. As in the previous case, we may assume that $r$ divides $|S|$ and it suffices to show that $\kappa(T,r) \geqs 3$. By \cite[Remark 5.4.7(b)]{KL}, $f$ divides $e$ and $n \geqs 7^{e/f}$. 

First assume $r$ divides $t-1$, so $i$ divides $e/f$ and thus $c=2i \leqs 2e/f$ and $e/f \geqs 3$. Since 
$$2e/f \geqs c \geqs \lceil \sqrt{n}/2 \rceil \geqs \lceil \sqrt{7^{e/f}}/2 \rceil$$
we deduce that $e/f=3$ is the only possibility, so $c=6$, $n \geqs 7^3$ and we clearly have $\kappa(T,r) \geqs 3$. Now assume $r$ divides $t^3+1$, so $i$ divides $6e/f$. If $i$ is odd then $r$ divides $q^{3e/f}-1$ and $q^{3e/f}+1$, which is absurd, so $i=c$ is even. If $c \leqs 2e/f$ then the previous argument shows that $e/f=3$, so $n \geqs 7^3$ and the result follows as above. Finally, suppose that $c=6e/f$. By the usual argument we deduce that $e/f \leqs 3$. If $e/f=3$ then $c=18$, $n \geqs 7^3$ and we see that $\kappa(T,r) \geqs 3$. Now assume $e/f=1$, so $c=6$ and $n \geqs 12$ (since $c \leqs n/2$). By inspecting \cite[Table A.49]{Lu}, noting that $p=3$, we deduce that $n \geqs 27$ and the desired result follows.

\vs

\noindent \emph{Case 4. $S = {}^2F_{4}(2^e)$}

\vs

Set $t=2^e$ and note that $|S| = t^{12}(t^6+1)(t^4-1)(t^3+1)(t-1)$, where $e \geqs 1$ is odd. As noted in \cite[Remark 5.4.7(b)]{KL}, $f$ divides $e$ and $n \geqs 26^{e/f}$, where $q=2^f$. As in the previous cases, there is at most one $S$-class of subgroups of order $r$ containing elements in $H_0 \setminus S$, so we may assume that $r$ divides $|S|$. Set $j = \Phi(r,t)$ and note that $j \in \{1,2,4,6,12\}$. 

First assume that $j=12$, so $r$ divides $t^4-t^2+1$. By considering the structure of the maximal tori of $S$ we deduce that $\kappa(S,r)=1$ so it suffices to show that $\kappa(T,r) \geqs 3$. Since $c \leqs 12e/f$ and $n \geqs 26^{e/f}$, it follows that $12e/f \geqs \lceil \sqrt{26^{e/f}}/2 \rceil$, so $e/f=1$, $c=12$ and $n \leqs 576$. 
By inspecting \cite[Table A.50]{Lu}, noting that $p=2$, we deduce that $n=26$ is the only possibility and thus $V$ is the minimal module for $S$. Since $V \not\cong V^{(z)}$ for any positive integer $z<f$, it follows that $H_0$ does not contain any field automorphisms, so $\kappa(H_0,r)=1$ and the result follows.

Next assume $j=6$, so $r$ divides $t^3+1$. Once again, we see that $\kappa(S,r)=1$. Since $c \leqs 6e/f$ and $n \geqs 26^{e/f}$ we deduce that $e/f=1$ is the only possibility, so $c=6$, $n \geqs 26$ and $\kappa(T,r) \geqs 3$ as required. If $j \leqs 2$ then $c \leqs 2e/f< \lceil \sqrt{26^{e/f}}/2\rceil$, so this case does not arise. 

Finally, let us assume $j=4$. Here $4e/f \geqs c \geqs \lceil \sqrt{26^{e/f}}/2\rceil$ and thus $e/f=1$, so $c=4$ and $n \leqs 64$. From \cite[Table A.50]{Lu}, we deduce that $n=26$ and thus $V$ is the minimal module for $S$. Set $\bar{S} = F_4(K)$ and $\bar{V} = L(\l_1)$ (or $L(\l_4)$), so $\bar{V}$ is a minimal module. From the proof of \cite[Lemma 7.4]{Burr}, we see that $\nu(y) \geqs 8$ for all nontrivial semisimple elements $y \in \bar{S}$ (with respect to the action of $\bar{S}$ on $\bar{V}$). We conclude that every element $x \in T$ of order $r$ with $\nu(x)=4$ is a derangement.
\end{proof}

This completes the proof of Proposition \ref{p:c}. By combining this result with Corollary \ref{p:star} and Propositions \ref{p:a}, \ref{p:b} and \ref{p:thm1}, we conclude that the proof of Theorem \ref{t:main} is complete.

\end{document}